\documentclass{article}

\newcommand*{\ev}{\mathrm{ev}}
\newcommand*{\Sint}{\int\limits_{\mathclap{S^{d-1}}}\!}

\newcommand{\inserttitle}{Fixed-strength spherical designs}

\usepackage{setspace}
\renewcommand*{\thefootnote}{\fnsymbol{footnote}}
\usepackage[T1]{fontenc}

\newif\ifsitem
\newcommand{\setupstar}{%
  \global\sitemfalse
  \let\origmakelabel\makelabel
  \def\sitem{\global\sitemtrue\item}
  \def\makelabel##1{%
    \origmakelabel{\ifsitem\llap{\raisebox{0.17ex}{$\mathbf{\ast}$}\:}\fi##1}%
    \global\sitemfalse}%
} % prepends a star to problem numbers

\usepackage[letterpaper, portrait, margin=1in, footnotesep=1.5\baselineskip]{geometry}
\usepackage[nocompress]{cite}
\usepackage{booktabs}
\usepackage[dvipsnames]{xcolor}
\usepackage{todonotes}
\usepackage{contour}
    \contourlength{0.5pt}
\usepackage{tocloft}

\usepackage{imakeidx}
    \makeindex
\usepackage{environ} % for wrapping environments in macros
\usepackage[framemethod=tikz]{mdframed}
\usepackage{wrapfig}
\usepackage{mathtools}
\usepackage{pgfornament}
\usepackage{fourier-orns} % for warning symbol
\usepackage{stmaryrd}
\usepackage{halloweenmath}
\usepackage{faktor}
\usepackage[shortlabels]{enumitem}
    \newlength{\circlabelwidth}
    \setlist{nosep}
    \setlist[enumerate]{label=\textup{\arabic*.}}
    \newlist{subprob}{enumerate}{2}
        \setlist[subprob,1]{label={(\roman*)}}
        \setlist[subprob,2]{label={(\arabic*)}}
    \setlist[itemize]{labelindent=10pt,labelwidth=\circlabelwidth,leftmargin=!,label=$\circ$}
    \newlist{problems}{enumerate}{3}
        \setlist[problems,1]{before=\setupstar,label=\textup{\arabic*.}, itemsep=2pt, topsep=8pt,ref=\textup{\arabic*}}
        \setlist[problems,2]{before=\setupstar,label=(\alph*),parsep=0pt}
        \setlist[problems,3]{before=\setupstar,label=(\roman*),parsep=0pt}

\usepackage{fancyhdr}
\usepackage{dopestyle}

\makeatletter
    \renewcommand\@makefntext[1]{\leftskip=0em\hskip-0em\@makefnmark\,#1}
\makeatother

\surroundwithmdframed[skipabove=0.5\baselineskip, skipbelow=0.5\baselineskip, leftmargin=3pt, rightmargin=3pt, innerleftmargin=7pt, innerrightmargin=7pt, roundcorner=10pt, linewidth=2pt, linecolor=red!40, backgroundcolor=red!5]{headsup}

\surroundwithmdframed[topline=false, bottomline=false, innertopmargin=0pt, innerbottommargin=0pt, innerleftmargin=2pt, innerrightmargin=2pt, linewidth=0.2mm]{quote}

\surroundwithmdframed[topline=false, bottomline=false, innertopmargin=2.5pt, innerbottommargin=2.5pt, innerleftmargin=-10pt, leftmargin=-10pt, innerrightmargin=-10pt, rightmargin=7.7pt, linewidth=0.4mm]{quote}

\NewEnviron{method}{
    \parbox{\textwidth}{
        \textbf{\textsc{Method}}
        \begin{mdframed}[innerleftmargin=4pt,innerrightmargin=4pt,skipabove=3pt,skipbelow=0pt]
            \BODY
        \end{mdframed}
    }
}

\theoremstyle{itcaps}
\newtheorem{theorem}{Theorem}[section]
\newtheorem{corollary}[theorem]{Corollary}
\newtheorem{proposition}[theorem]{Proposition}
\newtheorem{lemma}[theorem]{Lemma}

\newtheorem*{theorem*}{Theorem}
\newtheorem*{corollary*}{Corollary}
\newtheorem{mtheorem}{Theorem} % main theorems
    \crefname{mtheorem}{theorem}{theorems}
    \crefname{mtheorem}{Theorem}{Theorems}
\newtheorem{mcorollary}[mtheorem]{Corollary} % corollary in intro
    \crefname{mcorollary}{corollary}{corollaries}
    \Crefname{mcorollary}{Corollary}{Corollaries}
\theoremstyle{solved}

\theoremstyle{caps}
\newtheorem{definition}[theorem]{Definition}
\newtheorem{exampleprimitive}[theorem]{Example}

    \crefname{exercise}{Exercise}{Exercises}

\newtheorem{question}{Question}
\newtheorem{problem}[question]{Problem}

\theoremstyle{remark}

\numberwithin{equation}{section}

\newcommand*{\newword}[2][]{\emph{#2}\index{%
    \ifx&#1&%
       #2%
    \else%
       #1%
    \fi}%
} % formats & indexes new words
\newcommand*{\oldword}[2][]{#2\index{%
    \ifx&#1&%
       #2%
    \else%
       #1%
    \fi}%
} % formats & indexes another instance of a previously indexed word

\newcommand*{\F}{\mathbb{F}}
\newcommand*{\E}{\mathbb{E}}

\newcommand{\tensor}{\otimes}

\DeclareMathOperator{\Tr}{Tr}

\DeclareMathOperator{\rank}{rank}

\let\phi\varphi

\let\epsilon\varepsilon

\let\oldchi\chi
\renewcommand{\chi}{\raisebox{1pt}{$\oldchi$}}

\title{\inserttitle}
\date{}

\allowdisplaybreaks
\begin{document}

% display math spaces
\setlength{\abovedisplayskip}{6pt plus 2pt minus 4pt}
\setlength{\abovedisplayshortskip}{1pt plus 3pt}
\setlength{\belowdisplayskip}{6pt plus 2pt minus 4pt}
\setlength{\belowdisplayshortskip}{6pt plus 2pt minus 2pt}

\setlength{\parindent}{10pt}

\newcommand*{\red}[1]{\textcolor{red}{#1}}

\thispagestyle{plain}

% make the title w/ bars above and below
\makeatletter

\null
\vspace{-1.5em}

\noindent
    \hfill%
    \vbox{%
        \hsize0.9\textwidth
        \linewidth\hsize
        {
          \hrule height 2\p@
          \vskip 0.25in
          \vskip -\parskip%
        }
        \centering
        {\LARGE\sc \inserttitle\par}
        {
          \vskip 0.26in
          \vskip -\parskip
          \hrule height 2\p@
          \vskip 0.20in%
        }
    }%
    \hfill\null%
\makeatother

{\centering
    \large\scshape Travis Dillon\par
    \vspace{0.25in}
}

\begin{abstract}
    \setstretch{0.95}
    \noindent
    A \emph{spherical $t$-design} is a finite subset $X$ of the unit sphere such that every polynomial of degree at most $t$ has the same average over $X$ as it does over the entire sphere. Determining the minimum possible size of spherical designs, especially in a fixed dimension as $t \to \infty$, has been an important research topic for several decades. This paper presents results on the complementary asymptotic regime, where $t$ is fixed and the dimension tends to infinity. The main results in this paper are (1) a construction of smaller spherical designs via an explicit connection to Gaussian designs and (2) the exact order of magnitude of minimal-size signed $t$-designs, which is significantly smaller than predicted by a typical degrees-of-freedom heuristic. We also establish a method to ``project'' spherical designs between dimensions, prove a variety of results on approximate designs, and construct new $t$-wise independent subsets of $\{1,2,\dots,q\}^d$ which may be of independent interest. To achieve these results, we combine techniques from algebra, geometry, probability, representation theory, and optimization.
\end{abstract}
\vspace{0.15in}

\section{Introduction}

One significant focus in discrete geometry is the study of structured and optimal point arrangements. Finding point sets that minimize energy, form efficient packings or coverings, maximize the number of unit distances, or avoid convex sets, for example, each comprise a significant and long-standing research program in the area \cite{research-problems-geometry}. Many other famous problems in discrete geometry are point arrangement problems in disguise: The famous equiangular lines problem, for example, corresponds to finding a regular simplex with many vertices in real projective space; the sphere kissing problem corresponds to packing points on a sphere. Spherical designs, the focus of this paper, are point sets that are uniformly distributed according to polynomial test functions.

\begin{definition}\label{def:spherical-design}
    Let $\mu$ denote the Lebesgue measure on the unit sphere $S^{d-1}$, normalized so that $\mu(S^{d-1}) = 1$. A set $X \subseteq S^{d-1}$ is called a \emph{spherical $t$-design} (or \emph{unweighted spherical $t$-design}) if 
    \begin{equation}\label{eq:unweighted-design-condition}
        \frac{1}{|X|} \sum_{x \in X} f(x)
        = \int\limits_{\mathclap{S^{d-1}}} f\, d\mu
    \end{equation}
    for every polynomial $f$ of total degree at most $t$. A \emph{weighted spherical $t$-design} is the set $X$ together with a weight function $w\colon X \to \R_{>0}$ such that
    \begin{equation}\label{eq:weighted-design-condition}
        \sum_{x \in X} w(x)\, f(x)
        = \int\limits_{\mathclap{S^{d-1}}} f\, d\mu,
    \end{equation}
    again for every polynomial $f$ of total degree at most $t$. If $w$ may take negative values, then $(X,w)$ is called a \emph{signed} design. The parameter $t$ is called the \emph{strength} of the design.
\end{definition}

Like many fundamental topics in discrete geometry, spherical designs have strong connections to a broad range of mathematics: numerical analysis \cite{efficient-designs}, optimization \cite{delsarte-spherical-codes-designs}, number theory and geometry \cite{spherical-designs-survey}, geometric and algebraic combinatorics \cite{spherical-designs-alg-survey}, and, of course, other fundamental problems in discrete geometry \cite{universally-optimal-sphere}. Moreover, from the perspective of association schemes, spherical designs are a continuous analogue of combinatorial designs \cite{design-theory-alg-comb}.

Naturally, it is easier to find weighted designs than unweighted designs. Indeed, it's not apparent that unweighted spherical designs of all strengths even exist! Certain well-known symmetric point sets are designs, but only for small strengths: the vertices of a regular icosohedron, for example, form a 5-design on $S^2$, while the vertices of a cross-polytope form a 3-design in any dimension. In a remarkable 1984 paper, Seymour and Zaslavsky proved the existence of spherical designs of all strengths:

\begin{theorem}[Seymour, Zaslavsky \cite{Seymour-Zaslavsky}]\label{thm:seymour-zaslavsky}
    For all positive integers $d$ and $t$, there is a number $N(d,t)$ such that for all $n \geq N(d,t)$, there is an unweighted spherical $t$-design in $\R^d$ with exactly $n$ points.
\end{theorem}

With this established, mathematical attention turned  in earnest toward determining the size of the smallest spherical designs. The earliest bounds on the sizes of spherical designs were established by Delsarte, Goethals, and Seidel \cite{delsarte-spherical-codes-designs} by extending Delsarte's linear programming method for association schemes and coding theory \cite{delsarte-thesis}.

\begin{theorem}[Delsarte, Goethals, Seidel \cite{delsarte-spherical-codes-designs}]\label{thm:delsarte-lower-bound}
    The number of points in a spherical $t$-design in $\R^d$ is at least
    \[
        \begin{cases}
            \displaystyle \binom{d+ \lfloor t/2 \rfloor - 1}{d-1} + \binom{d + \lfloor t/2 \rfloor - 2}{d-1} & \text{if $t$ is even.}\\[3ex]
            \displaystyle 2\binom{d+\lfloor t/2\rfloor - 1}{d-1} & \text{if $t$ is odd.}
        \end{cases}
    \]
\end{theorem}

This result highlights the fundamental role that parity plays in designs. Given $\alpha = (\alpha_1,\alpha_2,\dots,\alpha_d) \in \N_0^d$, we let $x^\alpha = x_1^{\alpha_1}x_2^{\alpha_2}\cdots x_d^{\alpha_d}$ and $|\alpha| = \alpha_1 + \cdots + \alpha_d$. If $|\alpha|$ is odd, then
\[
    \int\limits_{\mathclap{S^{d-1}}} x^\alpha\, d\mu
    = \int\limits_{\mathclap{S^{d-1}}} (-x)^{\alpha}\, d\mu
    = -\int\limits_{\mathclap{S^{d-1}}} x^{\alpha}\, d\mu,
\]
so $\int_{S^{d-1}} x^\alpha\,d\mu = 0$. For the same reason, $\sum_{y \in Y} y^{\alpha} = 0$ for any antipodally symmetric set. The upshot is that, if $X$ is a spherical $2t$-design, then $X \sqcup (-X)$ is a spherical $(2t+1)$-design. For this reason, we will sometimes state results only for even designs.

Designs whose size exactly meets the lower bound in \Cref{thm:delsarte-lower-bound} are called \emph{tight designs}. These point sets often have further structure and symmetry beyond the design condition and only exist for a few values of $t$; see \cite{spherical-designs-alg-survey} for further information and references.

As for constructions of spherical designs, a linear algebraic argument shows proves the existence of weighted spherical designs with $O_d(t^{d-1})$ points, which matches the lower bound in \Cref{thm:delsarte-lower-bound}. (See \Cref{thm:kane-spherical-upper-bound}.) The question is then to find small \emph{unweighted} designs, which is a much harder problem. Hardin and Sloane found the exact size of the smallest unweighted designs on $S^2$ for certain specific values of $t$ \cite{hardin-sloane}. The broader research question is to determine the order of magnitude of minimum-size spherical designs, for a fixed dimension $d$, as the strength $t\to \infty$.

In the early 1990s, several mathematicians rapidly reduced the upper bound: Wagner \cite{sph-designs-wagner} constructed a spherical design with $O_d(t^{C d^4})$ points, which Bajnok \cite{sph-designs-bajnok} improved the following year to $O_d(t^{C d^3})$, followed a year later by Korevaar and Meyers's bound \cite{sph-designs-korevaar} of $O_d(t^{(d^2 + d)/2})$. In 2013, Bondarenko, Radchenko and Viazovska used topological methods to show the existence of designs whose size matches the lower bound in \Cref{thm:delsarte-lower-bound}.

\begin{theorem}[Bondarenko, Radchenko, Viazovska \cite{optimal-designs-fix-dim}]\label{thm:viazovska}
    There are numbers $N(d,t) = O_d(t^{d-1})$ such that for any $n \geq N(d,t)$, there is an unweighted spherical $t$-design in $\R^d$ with $n$ points.
\end{theorem}

Later works extended this result to designs on manifolds \cite{designs-compact-manifolds,designs-compact-alg-manifolds} or addressed the same problem in a general topological setting \cite{kane-small-designs}.

Interestingly, there appears to be little published research on the opposing regime, which holds the strength fixed as the dimension tends to infinity. Some work explores properties of 3- or 5-designs (for example, \cite{boyvalenkov,boyvalenkov2,3-designs,3-designs-abelian}), and one paper focuses on odd-strength spherical designs with an odd number of points \cite{boyvalenkov99}, improving the lower bound in this case from roughly $2 d^{\lfloor t/2\rfloor}/e!$ (in \Cref{thm:delsarte-lower-bound}) to $(1+2^{1/t}) d^{\lfloor t/2\rfloor}/e!$. The goal of this paper is to investigate the fixed-strength regime in more generality.

The first surprising aspect of this problem is the difficulty of determining the optimal size even for \emph{weighted} designs. When the dimension is fixed, the linear-algebraic construction in \Cref{thm:kane-spherical-upper-bound} produces a design whose size is within a constant of the \Cref{thm:delsarte-lower-bound}'s lower bound. This changes dramatically when the strength is fixed: \Cref{thm:delsarte-lower-bound} claims that every $t$-design has $\Omega_t(d^{\lfloor t/2\rfloor})$ points, while the linear-algebraic construction produces a design with $O_t(d^t)$ points. Therefore, surprisingly, in the fixed-strength regime, determining the minimal size of a design is an interesting problem even for weighted or signed designs.

The first few results in this paper rely on establishing a connection between spherical designs and designs over Gaussian space. Any probability measure on $\R^d$ can replace $\mu$ in \Cref{def:spherical-design}, and each measure gives a different design problem. Despite the close relationship between the Gaussian and spherical measures and the existence of previous research on Gaussian designs \cite{tight-gaussian-4-designs,non-exist-tight-gaussian-6-designs}, there appears to be no published result that connects these design problems. This paper provides an explicit connection, showing how to transform a Gaussian design into a spherical design and vice versa.

One nice application of the connection between Gaussian and spherical designs is the ability to ``project'' spherical designs to lower dimensions. Many point arrangement problems are monotone in the dimension because of a natural embedding into higher dimensions. For example, a set of equiangular lines in $k$ dimensions is also equiangular in $n > k$ dimensions, and a kissing configuration of points on $S^k$ (in which the distance between each pair of points is exactly $1/2$) is also a kissing configuration on $S^n$ for $n > k$. Similarly, a spherical $t$-design on $S^{d-1}$ is already a $(t-1)$-design, as well.

However, a spherical design does not easily embed in a different dimension: a spherical design on $S^k$ will not correctly average the polynomial $x_{k+1}^2$ over $S^n$ for any $n > k$. As a result, it's not clear whether the minimal number of points in a $t$-design on $S^n$ is an increasing function of $n$. However, Gaussian designs on $\R^n$ naturally project onto Gaussian designs on $\R^k$. By transferring a spherical design to a Gaussian design, projecting, and transferring back, we can convert $t$-designs on $S^n$ into $t$-designs on $S^k$ for $k < n$, showing that the minimum size of a $t$-design on $S^d$ is ``almost monotone'' in $d$. (See \Cref{thm:spherical-design-projection} for the exact statement.)

The main reason for developing this connection, however, is to construct small spherical designs by constructing small Gaussian designs. The first main result of this paper does just that, constructing an unweighted design that establishes an upper bound which is not only explicit but even smaller than the upper bound of $O_t(d^t)$ for \emph{weighted} designs.

\begin{mtheorem}\label{thm:unweighted-gaussian-design}
    There is an unweighted Gaussian $t$-design in $\R^d$ with $O_t(d^{t-1})$ points.
\end{mtheorem}

The transfer principle between Gaussian and spherical designs then establishes the existence of correspondingly small spherical designs.

\begin{mcorollary}\label{thm:smaller-sph-designs}
    There is a weighted spherical $t$-design on $S^{d-1}$ with $O_t(d^{t-1})$ points, and there is a multiset with at most $O_t(d^{t-1})$ distinct points that forms an unweighted spherical $t$-design on $S^{d-1}$. 
\end{mcorollary}

Unfortunately, the conversion does not produce an unweighted spherical design. However, \Cref{thm:smaller-sph-designs} is still an improvement on the previous upper bound for weighted spherical $t$-designs. Moreover, not all is lost: In Section 4 of \cite{Seymour-Zaslavsky}, Seymour and Zaslavsky show how to convert a weighted $t$-design with $N$ points into a multiset with at most $N$ distinct points which forms an unweighted $t$-design. Because the weights of the points may be irrational, the process is not trivial; they apply the Inverse Function Theorem. Applying that same process to the weighted design in \Cref{thm:smaller-sph-designs} proves the second half of the statement.

One main ingredient in the proof of \Cref{thm:unweighted-gaussian-design} is a new bound on $t$-wise independent sets which is likely of independent interest. Roughly speaking, a subset $Y \subseteq \{1,2,\dots,q\}^d$ is $t$-wise independent if the multiset projection of $Y$ to any set of $t$ coordinates is uniform on $\{1,2,\dots,q\}^t$. (See \Cref{def:k-wise-independent} for a formal definition.) The case $q=2$ has received the most attention for its applications to derandomizing algorithms in computer science, but the problem remains interesting and has many applications when $q>2$, as well. Combinatorialists and statisticians have studied $t$-wise independent sets under the name \emph{orthogonal arrays} \cite{orthogonal-arrays-book,orthogonal-arrays-review,orthogonal-arrays-rao,orthogonal-designs-info-theory}, where there is a strong connection to mathematical coding theory and experimental design. Computer scientists use $t$-wise independent sets, sometimes called \emph{$t$-universal hash functions}, for designing efficient randomized algorithms and managing limited memory in algorithms \cite{linear-probing,hash-functions-set-equality,k-wise-ind-construction,testing-k-wise-ind}.

Since $t$-wise independent sets are so well-studied, it comes as no surprise that there are many constructions of these sets that are effective in different parameter ranges. The most common general constructions come from error-correcting codes, such as the binary BCH or Reed--Solomon codes. The application in this paper, however, is in the unusual regime where $q > t$, and we provide an alternate construction which has an advantage in this setting:

\begin{mtheorem}\label{thm:t-wise-ind}
    If $q$ is a prime power, then there is a $t$-wise independent subset of $\{1,2,\dots,q\}^d$ with at most $(8qd)^{t-1}$ elements.
\end{mtheorem}

\Cref{sec:open} has a more in-depth comparison of \Cref{thm:t-wise-ind} to previous constructions.

\Cref{thm:unweighted-gaussian-design} and \Cref{thm:smaller-sph-designs} are improvements on the upper bound (especially \Cref{thm:unweighted-gaussian-design},since there is no \emph{a priori} upper bound for unweighted designs), but their size remains far from the lower bound of \Cref{thm:delsarte-lower-bound}. What size should we expect a minimal $t$-design to have? One way to form a prediction for the size of Gaussian designs, say, is to compare degrees of freedom with constraints. A configuration $X$ of $N$ points in $\R^d$ has $Nd$ degrees of freedom, and \eqref{eq:weighted-design-condition} represents $\binom{d+t}{t}$ constraints, one for each monomial of total degree at most $t$. Heuristically, we might expect a $t$-design to exist as long as the degrees of freedom outnumber the constraints: when $Nd > \binom{d+t}{t} = \Theta_t(d^t)$, or $N = \Omega_t(d^{t-1})$. And that is indeed the size of the designs in \Cref{thm:unweighted-gaussian-design} and \Cref{thm:smaller-sph-designs}.

That heuristic predicts the minimum size correctly when the dimension is fixed, but it turns out to be misleading for fixed-strength designs. There is nothing about the heuristic specific to weighted designs: The same reasoning holds for signed designs. However, it fails spectacularly to predict their size:

\begin{mtheorem}\label{thm:optimal-signed-designs}
    For every $t$, there are signed spherical and Gaussian $t$-designs with $O_t(d^{\lfloor t/2\rfloor})$ points.
\end{mtheorem}

By \Cref{thm:delsarte-lower-bound}, this is optimal up to the constant depending on $t$. This is genuinely surprising, because it means that, for fixed-strength designs, the constraints are not independent. They collude somehow behind the scenes. In fact, \Cref{thm:optimal-signed-designs} is an instance of a general phenomenon: The same bound holds for signed designs on any measure that is symmetric with respect to coordinate permutations and reflections (see \Cref{thm:signed-symmetric-measures}). All of which impugns the credibility of the degrees-of-freedom heuristic, making it hard to predict the true size of minimal weighted or unweighted $t$-designs.

The last collection of results in this paper addresses \emph{approximate} designs. Although approximate designs on the space of unitary matrices are well-studied in the quantum computing literature, their spherical counterparts seem unaddressed. Motivated by the research in unitary designs, we say that a weighted set of points $(X,w)$ is an \emph{$\epsilon$-approximate tensor $t$-design} if
\[
    \Big\lVert \sum_{x \in X} w(x)\,x^{\tensor t} - \int\limits_{\mathclap{S^{d-1}}} v^{\tensor t} \,d\mu(v) \Big\rVert_2 \leq \epsilon.
\]
Our main result in the final section is an asymptotically optimal determination of the minimum size of approximate designs.

\begin{mtheorem}\label{thm:approx-tensor-des}
    There is an $\epsilon$-approximate tensor $t$-design on $S^{d-1}$ with $\epsilon^{-2}$ points, and any such design has at least $\epsilon^{-2} - o_{d\to\infty}(1)$ points.
\end{mtheorem}

That section also proposes a non-equivalent definition of approximate designs which is motivated by numerical approximation to integrals. We prove lower bounds for this type of approximate design by modifying the linear programming method of Delsarte, Goethals, and Seidel \cite{delsarte-spherical-codes-designs} to accommodate approximation. Upper bounds for both types of approximate designs are derived using the probabilistic method.

% An interesting aspect of these results is the use of a recent ``no-dimensional'' version of Carath\'eodory's theorem (\Cref{thm:no-dim-caratheodory}) to construct approximate designs. This theorem has mainly been employed by convex and combinatorial geometers, but its use in this paper suggests there may be applications in other areas, as well.

The rest of the paper is organized as follows: \Cref{sec:history} outlines previously known constructions and bounds for spherical designs and provides a new, short proof of \Cref{thm:delsarte-lower-bound}. \Cref{sec:gaussian-designs} provides explicit conversions between spherical and Gaussian designs, in preparation for \Cref{sec:smaller-weighted-designs}, in which we prove \Cref{thm:unweighted-gaussian-design,thm:t-wise-ind} and \Cref{thm:smaller-sph-designs}. We prove \Cref{thm:optimal-signed-designs} in \Cref{sec:signed-designs}. The results on approximate designs, including \Cref{thm:approx-tensor-des}, appear in \Cref{sec:approximate-designs}. \Cref{sec:open} concludes the paper with a collection of open problems and further research directions.

\section{Prior bounds for designs}\label{sec:history}

\subsection{Lower bound}

The first lower bounds on the sizes of designs come from Delsarte, Goethals, and Seidel's seminal paper \cite{delsarte-spherical-codes-designs}. Their proof of \Cref{thm:delsarte-lower-bound} uses special functions and representation theory of functions on the sphere to prove that any spherical $2t$-design in $\R^d$ has at least
\[
    \binom{d+t-1}{d-1} + \binom{d+t-2}{d-1} = O_t(d^t)
\]
points. We will modify their argument to prove bounds on approximate designs in \Cref{sec:approximate-designs}. Here, we give a new, concise proof of this result that has the advantage of also working for signed designs, simplifying the linear-algebraic approach in \cite[Proposition 2.1]{designs-compact-alg-manifolds}.

Given a measure $\nu$ on $\R^d$, we let $\mathcal P_t^\nu$ denote the vector space of polynomial functions on the support of $\nu$. (Since elements of $\mathcal P_t^\nu$ are \emph{functions}, not the polynomials themselves, they may have multiple representations as polynomials. For example, the polynomials $x_1^2 + \cdots + x_d^2$ and $1$ represent the same element of $\mathcal P_t^\mu$, since the support of $\mu$ is the unit sphere.)

\begin{proposition}\label{thm:weighted-design-lower-bound}
    Any signed $2t$-design for a probability measure $\nu$ has at least $\dim(\mathcal P_t^\nu)$ points.
\end{proposition}
\begin{proof}
    Let $(X,w)$ be an signed $2t$-design. We claim that the linear transformation $\phi\colon \mathcal P_{t}^\nu \to \R^X$ by $\phi(f) = \big(f(x)\big)_{x \in X}$ is injective. Indeed, if $\phi(f) = \phi(g)$, then
    \[
        \int\limits_{\mathclap{S^{d-1}}} (f-g)^2\, d\nu = \sum_{x \in X} w(x)\, \big(f(x) - g(x)\big)^2 = 0,
    \]
    so $f = g$. Therefore $|X| \geq \dim(\mathcal P_t^\nu)$.
\end{proof}

The beginning of Section \ref{sec:gegenbauer-appendix} outlines a proof that $\dim(\mathcal P_t^\mu) = \binom{d+t-1}{d-1} + \binom{d+t-2}{d-1}$, which proves the lower bound for spherical designs. The support of the Gaussian probability measure $\rho$ is all of $\R^d$, so every polynomial represents a different function; therefore any Gaussian $t$-design on $\R^d$ has at least $\dim(\mathcal P_t^\rho) = \binom{d+t}{t}$ points.

\subsection{Upper bounds}

In this section we review a general upper bound for designs from Kane's paper on design problems \cite{kane-small-designs}. Kane's paper addresses designs in a very general setting and is mainly focused on producing unweighted designs. However, Lemma 3 in \cite{kane-small-designs} provides a linear-algebraic upper bound for weighted designs that complements the bound in \Cref{thm:weighted-design-lower-bound}.

\begin{proposition}[Kane {\cite[Lemma 3]{kane-small-designs}}]\label{thm:weighted-design-upper-bound}
    There is a weighted $2t$-design with at most $\dim(\mathcal P_{2t}^\nu)$ points.
\end{proposition}

Kane's linear-algebraic construction provides the baseline to improve upon in this paper. For spherical and Gaussian designs, it gives:

\begin{corollary}\label{thm:kane-spherical-upper-bound}
    There is a weighted spherical $2t$-design in $\R^d$ with at most $\binom{d+2t-1}{d-1} + \binom{d+2t-2}{d-1} = O_t(d^{2t})$ points and a weighted Gaussian $2t$-design with at most $\binom{d+2t}{2t} = O_t(d^{2t})$ points.
\end{corollary}

An important special case of Kane's result, which we will refer to later, is probability measures on $\R$:

\begin{proposition}\label{thm:moment-approximation}
    For any probability distribution $\nu$ on $\R$ with connected support and any integer $t\geq 0$, there is a probability distribution on at most $t+1$ points so that the first $t$ moments of the two distributions are equal.
\end{proposition}

On a related note, Seymour and Zaslavsky actually proved the existence not just of spherical designs, but of designs for any nice enough measure \cite{Seymour-Zaslavsky}. Applied to one-dimensional probability measures, it gives an unweighted version of \Cref{thm:moment-approximation}, though the size of the averaging set may be much larger.

\begin{proposition}\label{thm:moment-approximation-unweighted}
    For any probability distribution $\nu$ on $\R$ with connected support and any integer $t\geq 0$, there is an $N > 0$ such that: for any $n > N$, there is a finite set $Y$ of $n$ real numbers so that the first $t$ moments of $\nu$ are equal to the first $t$ moments of the uniform distribution on $Y$.
\end{proposition}

\subsection{Optimal constructions for small \texorpdfstring{$t$}{t}}

Although there are not many constructions of fixed-strength $t$-designs, there are some for 2- and 4-designs which meet the lower bound of \Cref{thm:delsarte-lower-bound}.

The standard basis vectors and their negations (which together form the vertices of a cross-polytope), form a spherical 2-design in $\R^d$ for every $d$, which can be confirmed by calculating the moments of the point set and comparing to the moments of the sphere.

In a 1982 paper \cite{levenstein-codes}, Leven\v{s}te\u{\i}n implicitly constructed an unweighted spherical 4-design with $d(d+2)$ points whenever $d$ is a power of $4$, using the binary Kerdock codes (see \cite{kerdock} for an explanation of the Kerdock codes). Additionally, Noga Alon and Hung-Hsun Hans Yu constructed a 4-design with $4d(d+2)$ points whenever $d$ is a power of $2$, using $2$-wise independent subsets of $\{\pm 1\}^d$ \cite{hans-communication}.

\section{Converting spherical and gaussian designs}\label{sec:gaussian-designs}

In this section, we show the connection between Gaussian and spherical designs and how to obtain either of these designs from the other. We also use this connection to ``project'' a spherical design to lower-dimensional spheres.

We begin with an explicit definition of Gaussian designs. For the rest of the paper, $\rho$ denotes the Gaussian probability measure on $\R^d$ given by $d\rho = e^{-\pi |x|^2}\, dx$.

\begin{definition}
    A set $X \subseteq \R^d$ and a weight function $w\colon X \to \R_{>0}$ are together called a \emph{weighted Gaussian $t$-design} if for every polynomial $f$ of degree at most $t$,
    \[
        \int\limits_{\mathclap{\R^d}} f(x)\, d\rho = \sum_{x \in X} w(x) f(x).
    \]
    If $w(x) = 1/|X|$ for each $x \in X$, then $X$ is called an \emph{unweighted} Gaussian $t$-design; if $w\colon X \to \R$, the design is called \emph{signed}.
\end{definition}

The key connection between the spherical and Gaussian probability measures is that, for any homogeneous polynomial $f$ of degree $k$,
\begin{equation}\label{eq:spherical-to-gaussian-integral}
    \int\limits_{\mathclap{\R^d}} f\, d\rho
    = \sigma_d \cdot \Big(\int\limits_{\mathclap{S^{d-1}}} f\, d\mu \Big) \Big(\int_0^\infty r^{k + d - 1} e^{-\pi r^2}\, dr\Big),
\end{equation}
where $\sigma_d$ is a constant (explicitly, $\sigma_d = 2\pi^{d/2}/\,\Gamma(d/2)$).

\subsection{Producing Gaussian designs from spherical, and vice versa}

\begin{proposition}\label{thm:spherical-to-Gaussian}
    If there is a weighted spherical $t$-design of size $N$ in $\R^d$, then there is a weighted Gaussian $t$-design in $\R^d$ with at most $(t+1)N$ points.
\end{proposition}
\begin{proof}
    Let $(X,w)$ be a weighted spherical $t$-design of $N$ points. As it turns out, if $P$ is a homogeneous polynomial and $\int_{S^{d-1}} P\,d\mu = 0$, then $(X,w)$ correctly averages $P$ over Gaussian space, as well. We will first prove that assertion, and then adjust $(X,w)$ so that it correctly averages the remaining polynomials of degree at most $t$ over Gaussian space.
    
    For any homogeneous polynomial $P$ that satisfies $\int_{S^{d-1}} P\, d\mu = 0$, we have
    \[
        \int\limits_{\mathclap{\R^d}} P\, d\rho
        = \sigma_d \cdot \Big(\int\limits_{\mathclap{S^{d-1}}} P\, d\mu\Big) \Big(\int_0^\infty r^{\deg P + d - 1} e^{- \pi r^2}\, dr\Big)
        = 0,
    \]
    so
    \[
        \int\limits_{\mathclap{\R^d}} P\, d\rho
        = 0
        = \int\limits_{\mathclap{S^{d-1}}} P\, d\mu
        = \sum_{x\in X} w(x)\, P(x).
    \]
    Moreover, for such a $P$ and any any $r > 0$, we have
    \[
        \sum_{x \in X} w(x)\, P(rx)
        = r^{\deg P} \sum_{x \in X} w(x)\, P(x)
        = r^{\deg P} \int\limits_{\mathclap{S^{d-1}}} P\, d\mu
        = 0.
    \]

    This motivates our strategy to find real numbers $r_1,\dots,r_{k}$ such that $\hat X = \bigcup_{i=1}^{t+1} r_i X$ is a Gaussian design. (Here, $r_i$ is a scaling factor, so $r_i X = \{ r_i x : x \in X\}$.) By using scaled copies of $X$, any homogeneous function $f$ for which $\sum_{x \in X} w(x) f(x) = 0$ will also have $\sum_{x \in \hat X} w(x) f(x) = 0$. This maintains the averages that are already correct.

    We now make use of a convenient basis for the vector space $\mathcal Q_k$ of all polynomials of degree $k$. (Recall that $\mathcal P_k^\mu$ is the vector space of polynomial \emph{functions} on $S^{d-1}$, so $1$ and $x_1^2 + \cdots + x_d^2$ represent the same element in $\mathcal P_k^\mu$ but different elements of $\mathcal Q_k$.) The vector space $\mathcal Q_k$ decomposes as
    \[
        \mathcal Q_k = \bigoplus_{\substack{i,j \geq 0 \\[2pt] i+2j \leq k}} \mathcal W_i\cdot |x|^{2j},
    \]
    where $\mathcal W_i$ is the set of homogeneous harmonic polynomials of degree $i$.\footnote{A polynomial $f$ is \emph{harmonic} if $\Delta f \equiv 0$, where $\Delta = \frac{\partial^2}{\partial x_1^2} + \cdots + \frac{\partial^2}{\partial x_d^2}$.} (Though not stated explicitly, the proof of Lemma 3.1 in \cite{henry-fourier-notes}, as a side effect, also proves this decomposition.) Homogeneous harmonic polynomials of different degrees are orthogonal, so in particular any $f \in \mathcal W_i$ with $i \geq 1$ is orthogonal to the constant function, which is just another way of saying that $\int_{S^{d-1}} f\,d\mu = 0$. Thus $\int_{S^{d-1}} f\,d\mu = 0$ for any $f \in \mathcal W_i \cdot \lvert x\rvert^{2j}$ with $i \geq 1$ and $j\geq 0$. This means that both $\int_{\R^d} f\,d\rho$ and $\sum_{x \in X} w(x)\, f(x)$ equal 0 for any $f \in \mathcal W_i \cdot \lvert x\rvert^{2j}$ with $i \geq 1$ and $j\geq 0$. Therefore, we only need to choose the $r_i$ so that the polynomials $|x|^{2j}$ average correctly.
    
    Let $\nu$ be the probability measure in $\R_{\geq 0}$ given by $d\nu = \sigma_d r^{d-1} e^{-\pi r^2}dr$. By \Cref{thm:moment-approximation}, there are real numbers $r_1,r_2,\dots,r_{t+1} \in \R$ that form a $t$-design for $\nu$ with some weights $\beta_1,\dots,\beta_{t+1}$. Then for any $0 \leq k \leq t/2$,
    \[
        \int\limits_{\mathclap{\R^d}} |x|^{2k}\,d\rho
        = \sigma_d \Big(\int\limits_{\mathclap{S^{d-1}}}  1\, d\mu\Big) \Big( \int_0^\infty r^{2k+d-1} e^{-\pi r^2}\,dr\Big)
        = \sum_{i=1}^{t+1} \beta_i\, r_i^{2k}
        = \sum_{i=1}^{t+1} \sum_{x \in X} \beta_i\, w(x)\, |r_i x|^{2k}.
    \]
    If we define the set $\hat X = \bigcup_{i=1}^{t+1} r_i X$ and the weight function $\hat w( r_i x) = \beta_i w(x)$, then $(\hat X, \hat w)$ is a Gaussian $t$-design with $(t+1)N$ points.
\end{proof}

% In the previous proof, if \Cref{thm:moment-approximation} is replaced by \Cref{thm:moment-approximation-unweighted}, we obtain an unweighted version:

% \begin{proposition}\label{thm:unweighted-gaussian-to-spherical}
%     For each $t$, there is a constant $c_t$ such that: If there is an unweighted spherical $t$-design of size $N$ in $\R^d$, then there is a unweighted Gaussian $t$-design in $\R^d$ with at most $c_t N$ points.
% \end{proposition}

Now we go in the opposite direction: constructing a spherical design from a Gaussian design.

\begin{proposition}\label{thm:Gaussian-to-spherical}
    If there is a weighted Gaussian $t$-design in $\R^d$ with $N$ points, there is a weighted spherical $t$-design in $\R^d$ with at most $2N$ points.
\end{proposition}
\begin{proof}
    Let $(X,w)$ be Gaussian $t$-design and enumerate the points as $X = \{x_1,\dots,x_N\}$. We'll show that the projection of $X$ onto the sphere, with a certain set of weights, correctly averages all even-degree monomials over the sphere. To construct a full design, we then symmetrize $X$.
    
    Let $s = 2 \lfloor t/2\rfloor$, the largest even integer $\leq t$. We will first establish the claim for monomials with total degree $s$ and then use that to prove the claim for all even-degree monomials. Suppose that $P$ is such a monomial, and $y_i = x_i/|x_i|$ and $r_i = |x_i|$. We have
    \[
        \sigma_d \cdot \Big(\int\limits_{\mathclap{S^{d-1}}} P\, d\mu\Big) \Big(\int_0^\infty r^{s + d - 1} e^{-\pi r^2}\, dr\Big)
        = \int\limits_{\mathclap{\R^d}} P\, d\rho
        = \sum_{i=1}^N w(x_i) r_i^s \, P(y_i).
    \]
    If we define $Y$ as the point set $\{y_1,\dots,y_N\} \subset S^{d-1}$, then from \eqref{eq:spherical-to-gaussian-integral},
    \[
        \int\limits_{\mathclap{S^{d-1}}} P\, d\mu
        =\frac{1}{\sigma_d \cdot \int_0^\infty r^{s + d - 1} e^{-\pi r^2}\, dr} \sum_{i=1}^N w(x_i) \, r_i^s \, P(y_i)
        = \sum_{i=1}^N \hat w(y_i)\, P(y_i)
    \]
    for the weight function
    \[
        \hat w(y_i) = \frac{w(x_i)\, r_i^s}{\sigma_d \cdot \int_0^\infty r^{s + d - 1} e^{-\pi r^2}\, dr}.
    \]
    For a monomial $\hat P$ of degree $2k < s$, the polynomial $\hat P \cdot \lvert x\rvert^{s-2k}$ is homogeneous of degree $s$ and takes the same values as $\hat P$ on $S^{d-1}$, so
    \[
        \int\limits_{\mathclap{S^{d-1}}} \hat P\, d\mu
        = \int\limits_{\mathclap{S^{d-1}}} \hat P(x)\, \lvert x \rvert^{s-2k}\, d\mu(x)
        = \sum_{i=1}^N \hat w(y_i)\, \hat P(y_i)\, |y_i|^{s-2k}
        =\sum_{i=1}^N \hat w(y_i)\, \hat P(y_i).
    \]
    Therefore $(Y,\hat w)$ correctly averages all even-degree monomials. The point set $\hat Y = Y \sqcup (-Y)$ with weight function $\frac{1}{2}(\hat w(y_i) + \hat w(-y_i)\big)$ correctly averages monomials of odd degree as well, so it is a spherical $t$-design with at most $2N$ points.
\end{proof}

Together, \Cref{thm:spherical-to-Gaussian} and \Cref{thm:Gaussian-to-spherical} show that, as $n\to\infty$ for fixed $t$, the growth rates of the smallest weighted spherical and Gaussian $t$-designs are the same.

\subsection{Projecting spherical designs}

We now use these results to ``project'' spherical designs to lower dimensions.

\begin{lemma}\label{thm:Gaussian-projection}
    The orthogonal projection of a weighted Gaussian $t$-design in $\R^d$ onto $\R^k$ is a Gaussian $t$-design in $\R^k$ with the same weights.
\end{lemma}
\begin{proof}
    Let $X$ be a Gaussian $t$-design in $\R^d$ and $\pi\colon \R^d \to \R^k$ be the orthogonal projection that deletes the last $d-k$ coordinates of a point. For any polynomial $P$ in $\R^k$, let $\tilde P$ be its extension to $\R^d$ given by $\tilde P(x) = P\big( \pi (x)\big)$. Then
    \[
        \sum_{x \in X} w(x)\, P\big( \pi (x)\big)
        = \sum_{x \in X} w(x)\, \tilde P(x)
        = \int\limits_{\mathclap{\R^d}} \tilde P\, d\rho
    \]
    by the fact that $X$ is a design, and using the fact that $\int_{-\infty}^\infty e^{-\pi |x|^2}\,dx = 1$, we have
    \[
        \int\limits_{\mathclap{\R^d}} \tilde P\, d\rho
        = \Big(\int\limits_{\mathclap{\R^k}} P(x_1,\dots,x_k)\, d\rho\Big)\, \Big( \int\limits_{\ \mathclap{\R^{d-k}}} e^{-\pi(x_{k+1}^2 + \cdots + x_d^2)}\, d\rho \Big)
        = \int\limits_{\mathclap{\R^k}} P\, d\rho.
    \]
    So $\big(\pi(X), \pi(w)\big)$ is a Gaussian $t$-design in $\R^k$.
\end{proof}

\begin{corollary}\label{thm:spherical-design-projection}
    If there is a weighted spherical $t$-design in $\R^d$ with $N$ points, then there is a weighted spherical $t$-design in $\R^k$, for each $k \leq d$, with at most $2(t+1)N$ points.
\end{corollary}
\begin{proof}
    Use \Cref{thm:spherical-to-Gaussian} to convert to a Gaussian design with $(t+1)N$ points; project the design to $\R^k$ using \Cref{thm:Gaussian-projection}; then convert back to a spherical design with $2(t+1)N$ points using \Cref{thm:Gaussian-to-spherical}.
\end{proof}

\Cref{thm:spherical-design-projection} combined with the Kerdock construction of 4-designs in \Cref{sec:history} shows that for \emph{every} dimension $d$, there is a weighted spherical 4-design in $S^{d-1}$ with at most $2 \cdot 5 \cdot (4d)(4d+2) < 160 d(d+1)$ points. (Simply take a 4-design in a dimension larger than $d$ which is a power of $4$ and project that to a spherical design in $\R^d$.)

\section{Smaller unweighted Gaussian designs}\label{sec:smaller-weighted-designs}

The aim of this section is to prove \Cref{thm:unweighted-gaussian-design,thm:t-wise-ind} and deduce \Cref{thm:smaller-sph-designs} from them. We'll start by formally defining $t$-wise independent sets and proving \Cref{thm:t-wise-ind}.

\begin{definition}\label{def:k-wise-independent}
    Let $A$ be a finite set and $X$ a multiset of vectors in $A^d$. For each $I \subseteq [d]$, let $X_I$ be the random variable obtained by choosing a uniform random vector in $X$ and restricting to the coordinates in $I$. If the distribution of $X_I$ is uniform on $A^{\lvert I \rvert}$ for every $I \subseteq [r]$ with $\lvert I \rvert \leq k$, then $X$ is called \emph{$k$-wise independent}.
\end{definition}

The idea for using $t$-wise independent sets to construct an unweighted Gaussian design comes from the fact that the Gaussian is a product measure: If a point set in $\R^d$ is $t$-wise independent, and the distribution along each coordinate is itself a 1-dimensional Gaussian $t$-design, then show that point set is a Gaussian $t$-design in $\R^d$.

So our goal will be to construct small $t$-wise independent sets. To do this, we first connect $t$-wise independence of sets to $t$-wise independence of vectors in $\F_q^n$. Then, we construct a set of $t$-wise independent vectors using the probabilistic method.

\begin{lemma}\label{thm:t-wise-linear-to-set}
    For each vector $x \in \F_q^r$, define the vector $\phi_x \in (\F_q)^{\F_q^r}$ by $\phi_x(y) = \langle x,y\rangle$. If $S \subseteq \F_q^r$ has the property that any collection of $t$ elements of $S$ is linearly independent, then the uniform distribution on $\{\!\{ \phi_x|_S : x \in \F_q^r\}\!\}$ is $t$-wise independent.
\end{lemma}
\begin{proof}
    Checking that the multiset is $t$-wise independent corresponds to fixing any $t$ elements of $y_1,\dots,y_t\in S$ and examining the restriction $\phi_x|_S$, which is the map $\psi\colon \F_q^r \to \F_q^t$ given by $\psi(x)_i = \langle x,y_i\rangle$. Since $y_1,\dots,y_t$ are linearly independent and $\psi$ is represented by a matrix whose $i$th row is $y_i$, we know that $\rank \psi = t$. So $\psi$ is surjective, and every element of $\F_q^t$ has a preimage of size $q^{r-t}$. In other words, the uniform distribution on $Y = \{y_i\}_{i=1}^t$ yields the uniform distribution on $\{\!\{ \phi_x|_Y : x \in \F_q^r\}\!\}$. As this holds for any subset $Y \subseteq S$ of size $t$, the uniform distribution on $\{\!\{ \phi_x|_S : x \in \F_q^r\}\!\}$ is $t$-wise independent.
\end{proof}

The set $\{\!\{\phi_x|_S : x \in \F_q^r\}\!\}$ has $q^r$ vectors, each with $d = |S|$ coordinates. So to find a small $t$-wise independent set relative to the number of coordinates, we want to maximize the size of $S$.

\begin{lemma}\label{thm:t-wise-linearly-indep-set}
    The vector space $\F_q^r$ contains a set of size $\frac{1}{8} q^{r/(t-1)-1}$ such that any linearly dependent subset has size at least $t+1$.
\end{lemma}
\begin{proof}
    We will prove the existence of this set probabilistically. Every linearly dependent subset of size $t$ can be written in the form $(Y,v)$, where $Y$ has $t-1$ vectors and $v \in \operatorname{span}(Y)$, so the number of linearly dependent subsets of size $t$ in $\F_q^r$ is at most $\binom{q^r}{t-1} q^{t-1}\leq q^{(r+1)(t-1)}$. Create a set $T$ by including each vector in $\F_q^r$ independently with probability $\alpha$. Then $\E\big[|T|\big] = \alpha q^r$ and the expected number of linearly dependent subsets of size $t$ in $T$ is at most $q^{(r+1)(t-1)} \alpha^t$. Delete all the vectors from each linear dependence of size $t$ to get a set $S$ with no such linear dependence and size $\E\big[|S|\big] \geq q^r \alpha - t q^{(r+1)(t-1)} \alpha^t$. Taking $\alpha = (\frac{1}{2t})^{1/(t-1)} q^{-(r+1) + r/(t-1)}$ yields $\E\big[|S|\big] \geq \frac{1}{8} q^{r/(t-1)-1}$.
\end{proof}

\begin{proof}[Proof of \Cref{thm:t-wise-ind}]
    Together, \Cref{thm:t-wise-linearly-indep-set,thm:t-wise-linear-to-set} construct a set of $m=q^r$ vectors with $d = \frac{1}{8q} m^{1/(t-1)}$ coordinates that is $t$-wise independent. Rearranging this equation, we see that this is a set of $m = (8qd)^{t-1}$ vectors lying in $\{1,2,\dots,q\}^d$ that are $t$-wise independent.
\end{proof}

We can now prove \Cref{thm:unweighted-gaussian-design} from \Cref{thm:t-wise-ind}.

\begin{proof}[Proof of \Cref{thm:unweighted-gaussian-design}]
    By \Cref{thm:moment-approximation-unweighted}, for some prime-power $q \in \N$, there is an unweighted Gaussian $t$-design $A = \{a_1,\dots,a_q\} \subset \R^1$. \Cref{thm:t-wise-ind} produces a set $X \subseteq A^d$ of at most $(8q d)^{t-1}$ vectors in $\R^d$ whose coordinates are $t$-wise independent. Since $A$ is a Gaussian $t$-design, the first $t$ moments of $A$ are the first $t$ moments of the Gaussian measure. So take any monomial $x_{i_1}^{\alpha_1} x_{i_2}^{\alpha_2} \cdots x_{i_r}^{\alpha_r}$ of degree at most $t$. Because $r \leq t$, the $t$-wise independence allows averages to distribute over the product, so
    \[
        \mathop{\E}_{x \in X} [x_{i_1}^{\alpha_1}\cdots x_{i_r}^{\alpha_r}]
        = \prod_{j=1}^r \mathop{\E}_{x\in X} [x_{i_j}^{\alpha_j}]
        = \prod_{j=1}^r \, \int\limits_{\R} x_{i_j}^{\alpha_j} e^{-\pi x_{i_j}^2}\, dx_{i_j}
        = \int\limits_{\mathclap{\R^d}} x_{i_1}^{\alpha_1}\cdots x_{i_r}^{\alpha_r}\, e^{-\pi |x|^2}\, dx\qedhere
    \]
\end{proof}

In fact, the set $X$ in the previous proof correctly averages every monomial $x_{i_1}^{\alpha_1} x_{i_2}^{\alpha_2} \cdots x_{i_t}^{\alpha_t}$ with at most $t$ variables as long as $\alpha_i \leq t$ for each $t$, which is a stronger condition than being a $t$-design. In any case, by applying \Cref{thm:Gaussian-to-spherical} to \Cref{thm:unweighted-gaussian-design}, we get a spherical $t$-design, as well.

\begin{corollary}
    There is a weighted spherical $t$-design in $\R^d$ with $O_t(d^{t-1})$ points.
\end{corollary}

%%%%%%%%%%%%%%%%%%%%%%%%%%%%%%%%

\section{Optimal signed designs}\label{sec:signed-designs}

In this section, we prove \Cref{thm:optimal-signed-designs} by constructing signed Gaussian and spherical designs whose sizes are within a multiplicative constant of the lower bound in \Cref{thm:delsarte-lower-bound}.

If $\alpha_i$ is odd, then $\int_{S^{d-1}} x^{\alpha}\,d\mu(x) = 0$, so any point set that is symmetric with respect to the coordinate negation $x_i \mapsto -x_i$ correctly averages the monomial $x^\alpha$. The idea of this section is to divide the monomials of degree at most $2t$ into two groups, those that have even degree in each variable and those that do not, and address the two groups in different ways.

The simplest approach using this idea would be to start with a point set $X$ that correctly averages all monomials $x^\alpha$ with $\alpha \in (2\N_0)^d$ and $\lvert \alpha\rvert \leq 2t$. Since there are $O(d^t)$ such polynomials, a slight modification of \Cref{thm:weighted-design-upper-bound} produces an averaging set for these monomials with $O(d^t)$ points. Then, we can take various coordinate negations of this set to make all the remaining monomials average to 0, as they do over $S^{d-1}$.

Given $\epsilon \in \{\pm 1\}^d$, let $\eta_\epsilon$ be the linear transformation defined by $x_i \mapsto \epsilon_i x_i$. The set $\bigcup_{\epsilon \in \{\pm 1\}^d} \eta_{\epsilon}(X)$ is certainly symmetric across each coordinate and therefore creates a $2t$-design, but one with $2^d |X|$ points, which is enormous. The idea would be to reduce the number of coordinate negations needed to create a coordinate-symmetric set. 

If every monomial $x^\alpha$ with $\lvert \alpha\rvert \leq 2t$ and $\alpha \notin (2\N_0)^d$ averages to 0 across $\eta_{\epsilon_1}(X) \cup \cdots \cup \eta_{\epsilon_m}(X)$ for every point set $X$, then for any $1 \leq r \leq 2t$ distinct values $i_1,\dots,i_r \in [m]$,
\begin{equation}\label{eq:mean-zero-reflections}
    \E_j\big[ \epsilon_j(i_1) \epsilon_j(i_2)\cdots \epsilon_j(i_r) \big] = 0.
\end{equation}
(The indices $i_1,\dots,i_r$ correspond to the coordinates in $\alpha$ that are odd.) Similarly, if \eqref{eq:mean-zero-reflections} is satisfied, then $\eta_{\epsilon_1}(X) \cup \cdots \cup \eta_{\epsilon_m}(X)$ is a $2t$-design, as long as $X$ correctly averages the monomials $x^{\alpha}$ with $\alpha \in (2\N_0)^d$. As it turns out, satisfying \eqref{eq:mean-zero-reflections} requires at least $\binom{d}{t}$ reflections, even if the reflections can be weighted according to a probability distribution.

\begin{proposition}[Sauermann \cite{sauermann-communication}]\label{thm:Sauermann}
    If $\epsilon_1,\dots,\epsilon_m \in \{\pm 1\}^d$ satisfy condition \eqref{eq:mean-zero-reflections} according to a probability distribution $\nu$ on $\{\epsilon_i\}_{i=1}^m$, then $m \geq \binom{d}{t}$.
\end{proposition}
\begin{proof}
    We define an $\binom{[d]}{t}\times \binom{[d]}{t}$ matrix $M$ by
    \begin{equation}\label{eq:sauermann-lemma}
        M_{S,T}
        = \E_{j\sim \nu} \Big[ \prod_{i\in S} \epsilon_j(i) \prod_{i\in T}\epsilon_j(i) \Big]
        =  \E_{j\sim \nu} \Big[ \prod_{i\in S \triangle T} \epsilon_j(i) \Big],
    \end{equation}
    where $\triangle$ denotes the symmetric difference. If $S \neq T$, then condition \eqref{eq:mean-zero-reflections} implies that $M_{S,T} = 0$; if $S=T$, then $M_{S,S} = 1$. So $M$ is the identity matrix, and $\rank(M) = \binom{d}{t}$. On the other hand, each matrix $M_j$ defined by $(M_j)_{S,T} := \prod_{i\in S} \epsilon_j(i) \prod_{i\in T}\epsilon_j(i)$ has rank 1, so $M = \E_j [M_j]$ has rank at most $m$. We conclude that $\binom{d}{t} = \rank(M) \leq m$.
\end{proof}

Thus, any design produced by this method has at least $\Omega_t(d^t |X|) = \Omega_t(d^{2t})$ points, which is no improvement on the existing upper bound at all.

% As a side note, any $k$-wise independent set in $\{\pm 1\}^d$ satisfies equation \eqref{eq:mean-zero-reflections}. Proposition 15.2.3 in \cite{alon-spencer} proves that any collection of $2t$-wise independent bits has $\Omega_t(d^t)$ vectors; \Cref{thm:Sauermann} shows that even condition \eqref{eq:mean-zero-reflections}, which is weaker, guarantees this lower bound.

To overcome this problem, we will choose $X$ more judiciously. If each point in $X$ is zero in many coordinates, then it has few images under coordinate negations, which means that the set generated from $X$ that is symmetric across all coordinates may be much smaller than in general.

The next proof uses this idea by starting with a family of coordinate-symmetric sets that each correctly average the monomials $\{ x^{\alpha} : \alpha \notin (2\N_0)^d\}$ and taking a weighted union of them to correctly average the remaining monomials.

\begin{proof}[Proof of \Cref{thm:optimal-signed-designs}]
     Before diving into the proof, here is a preview of what's to come. We'll start with a family of symmetric point sets $Y_t(a)$ parametrized by points $a \in \R^t$, and the signed $2t$-design will be formed as a weighted union of several different $Y_t(a)$'s. To find a good weighted union, we will transform the problem into a linear-algebraic one about the moment vectors of the Gaussian measure and the $Y_t(a)$, and then show that the Gaussian moment vector is in the span of the moment vectors of the $Y_t(a)$. The symmetry of the $Y_t(a)$ allows the argument to take place in a low-dimensional vector space, which results in a design with few points.
     
    Now to the specifics. We start by defining the $Y_t(a)$. Given $a \in \R^{t}$, we consider the set of images of $a_1e_1 + \cdots a_{t} e_{t} \in \R^d$ under coordinate permutation and negation, including multiplicity:
    \[
        \big\{\!\big\{ \sum_{i=1}^{t} \epsilon_i a_i e_{\sigma(i)} : \epsilon \in \{\pm 1\}^{t} \text{ and } \sigma \in S_d \big\}\!\big\}.
    \]
    This multiset has $2^t d!$ elements, and we define $Y_t(a)$ as the multiset obtained from this one by dividing the multiplicity of each element by $(d-t)!$. (All multiplicities in $Y_t(a)$ are integers because each element of the original multiset has at least $d-t$ zeros.) So $\lvert Y_t(a)\rvert < 2^t d^t$.
    
    Our goal is to find a weighted union of the $Y_t(a)$ that is a signed Gaussian $2t$-design: that is, a set $A \subset \R^t$ and a function $w\colon A \to \R$ such that
    \begin{equation}\label{eq:gluing-signed-design-condition}
        \sum_{a \in A} w(a)\ \sum_{\mathclap{y \in Y_t(a)}} y^\alpha
        = \int\limits_{\mathclap{\R^d}} y^\alpha\, d\rho
    \end{equation}
    for every $\alpha \in \N_0^d$ with $|\alpha|\leq t$. 
    
    To reformulate the problem linear-algebraically, we denote the moments of the $Y_t(a)$ and the Gaussian measure by
    \[
        b_\alpha(a) := \sum_{y \in Y_t(a)} y^{\alpha}
        \qquad \text{and} \qquad
        m_\alpha = \int_{\R^d} x^\alpha\, d\rho
    \]
    With these notations, \eqref{eq:gluing-signed-design-condition} says that there is a $w\colon A \to \R$ such that $m_\alpha = \sum_{a \in A} w(a) b_{\alpha}(a)$ for every $\alpha \in (\N_0)^d$ such that $|\alpha| \leq t$.

    For some $\alpha$, these conditions are always satisfied: If some coordinate of $\alpha$ is odd, then $b_\alpha(a) = 0$ for every $a \in \R^t$, because $Y_t(a)$ is symmetric under coordinate negations; and in this case, $m_\alpha = 0$ as well. Moreover, $b_\alpha(a)$ is invariant under permutations:
    \[
        b_{\sigma \cdot \alpha}(a)
        = \sum_{y \in Y_t(a)} y^{\sigma\cdot \alpha}
        =\sum_{y \in Y_t(a)} (\sigma^{-1}\cdot y)^{\alpha}
        = b_{\alpha}(\sigma^{-1} \cdot a)
        = b_{\alpha}(a).
    \]
    In short, then, we we can restrict our attention to one $\alpha$ from each orbit of $S_d$ among those $\alpha \in (2\N_0)^d$ with $|\alpha|\leq t$. 
    
    One such set of representatives is $\{2\alpha : \alpha \in P\}$, where $P$ is the collection of partitions of integers $\leq t$ into at most $d$ parts. (A \emph{partition} of $t$ is a set of positive integers that sum to $t$.) If we define
    \[
        m := (m_{2\alpha})_{\alpha \in P}
        \qquad \text{and} \qquad
        b(a) := \big(b_{2\alpha}(a)\big)_{\alpha \in P},
    \]
    both in $\R^P$, then \eqref{eq:gluing-signed-design-condition} is equivalent to the statement that $m \in \operatorname{span}\!\big(b(a) : a \in \R^t\big)$.

    Actually, we now prove the stronger statement that $\big\{ b(a) : a \in \R^{t}\big\}$ spans $\R^P$, by contradiction: We will first show (unconditionally) that the polynomials $b_{2\alpha}$ are linearly independent, and then that if they do not span $\R^P$, they are linearly dependent---a contradiction.
    
    Considered as a map $\R^t \to \R$, the function $b_{2\alpha}(a)$ is a polynomial in $t$ variables in which the degree sequences of the monomials are the permutations of $2\alpha$. Thus, if $\lvert \alpha\rvert = \lvert \beta\rvert$ but $\alpha \neq \beta$, the monomials in $b_{2\alpha}$ do not appear in $b_{2\beta}$, which shows independence.
    
    Now, for the purpose of a contradiction, suppose that
    \[
        \dim\!\Big(\! \Span\! \big( b(a) : a \in \R^{t} \big) \Big) < |P|.
    \]
    Thus $\big\{ b(a) : a \in \R^{t}\big\}$ lies in a $(|P|-1)$-dimensional subspace, so there is a nonzero vector $c \in \R^P$ which is orthogonal to all $b(a)$. Explicitly,
    \[
        \sum_{\alpha \in P} c_\alpha\, b_{2\alpha}(a) = 0
    \]
    for every $a \in \R^t$. But this can only be the case if $\sum_{\alpha \in P} c_\alpha b_\alpha$ is the zero polynomial, which is impossible because the polynomials $\{b_{2\alpha}\}_{\alpha \in P}$ are linearly independent.

    Therefore, $m = \sum_{a \in A} w(a) b(a)$ for some $A \subset \R^t$ of size at most $|P|$ and $w\colon A \to \R$. The set $X = \cup_{a \in A} Y_t(a)$ with the weight function $\hat w(y) = w(a)$ whenever $y \in Y_t(a)$ is a signed Gaussian $2t$-design. It has at most $|P|\cdot |Y_t(a)| \leq p_t 2^t d^t = O_t(d^t)$ points.
\end{proof}

The constant $p_t$ in the proof is fairly reasonable. A simple upper bound by the number of \emph{compositions} (sequences of positive integers that sum to $t$) shows that $p_t \leq 2^{t-1}$. Actually, $p_t$ is much smaller; Hardy and Ramanujan showed that $p_t = O( e^{\pi \sqrt{2t/3}})$ \cite{hardy-ramanujan}.

By applying \Cref{thm:Gaussian-to-spherical}, we get the corresponding result for spherical designs:

\begin{corollary}
    There is a signed spherical $2t$-design in $\R^d$ with $O_t(d^t)$ points.
\end{corollary}

In fact, this approach proves something notably stronger. Let $\mathcal P^{o}_{2t}$ denote the span of the monomials $x^\alpha$ in $\R^d$ in which either $|\alpha| \leq 2t$ or some component of $\alpha$ is odd. The design constructed in \Cref{thm:optimal-signed-designs} averages to 0 on \emph{any} monomial with an odd degree component, not just those with total degree $2t$. Therefore, the same proof actually shows that:

\begin{theorem}\label{thm:signed-all-odd}
    There are signed spherical and Gaussian $\mathcal P_{2t}^o$-designs in $\R^d$ with $O(d^t)$ points.
\end{theorem}

This is a strong statement, since $\mathcal P_{2t}^o$ is an infinite-dimensional vector space, and an indication that the monomials with all even degrees are the driving force behind the lower bound of $\Omega_t(d^t)$ for the size of a spherical $2t$-design.

In the proof of \Cref{thm:optimal-signed-designs}, we didn't make use of the particular properties of the Gaussian measure other than its symmetry under coordinate permutations and negations, so this result extends to an entire family of measures: 

\begin{theorem}\label{thm:signed-symmetric-measures}
    If $\nu$ is a measure on $\R^d$ that is symmetric with respect to coordinate permutations and negation, then there is a signed $\mathcal P_{2t}^o$-design for $\nu$ with at most $O_t(d^t)$ points.
\end{theorem}

\section{Approximate designs}\label{sec:approximate-designs}

This section proposes two definitions of approximate designs and proves bounds on their sizes. \Cref{sec:L2-approximate-designs} introduces a definition of approximation relative to polynomial test functions, while \Cref{sec:tensor-approximate-designs} introduces a definition that parallels those for approximate unitary designs and proves \Cref{thm:approx-tensor-des}.

The lower bounds for the two types of approximate designs are proven in entirely different ways, but constructions for both are obtained using the probabilistic method.

% In both sections, the construction of small approximate designs relies on a version of Carath\'eodory's theorem without dimension. We let $\operatorname{cent}(Y)$ denote the centroid of a point set $Y$ and $\diam(X)$ denote the diameter of $X$ (the largest pairwise distance between points in $X$).

% \begin{theorem}[Adiprasito, B\'ar\'any, Mustafa, Terpai {\cite[Theorem 1.1]{no-dim-helly}}]\label{thm:no-dim-caratheodory}
%     If $X \subseteq \R^d$ and $y \in \conv(X)$, then for each $k \in \N$ there is a multiset $Y$ of $k$ points from $X$ for which
%     \[
%         \big\lVert y - \operatorname{cent}(Y)\big\rVert 
%         \leq \frac{\operatorname{diam}(X)}{\sqrt{2k}}.
%     \]
% \end{theorem}

% The idea for both constructions is to follow the proof of \Cref{thm:weighted-design-upper-bound}, replacing Carath\'eodory's theorem with \Cref{thm:no-dim-caratheodory}.\footnote{\Cref{thm:no-dim-caratheodory}, as stated in \cite{no-dim-helly}, claims only that the distance between $y$ and $\conv(Y)$ is at most $\diam(X)/\sqrt{2k}$, but their proof of this theorem guarantees the stronger conclusion stated here.}

\subsection{\texorpdfstring{$\MakeUppercase L^2$}{L2}-approximate designs}\label{sec:L2-approximate-designs}

Intuitively, an approximate design should satisfy $\sum_{x \in X} f(x) \approx \int f\, d\mu$. If we scale $f$ by a constant, then the error in the approximation will scale as well, so it makes sense to measure the error of the approximation relative to the norm of $f$:

\begin{definition}\label{def:L2-approx-designs}
    A set $X$ is an $\epsilon$-approximate spherical $t$-design if 
    \begin{equation}\label{eq:L2-approx-design}
        \left\lvert \frac{1}{|X|} \sum_{x \in X} f(x) - \int\limits_{\mathclap{S^{d-1}}} f(x)\, dx\right\rvert
        \leq \epsilon \lVert f\rVert_2
    \end{equation}
    for every polynomial $f$ of degree at most $t$, where $\lVert f\rVert_2$ is the $L^2$ norm $(\int\! f^2)^{1/2}$.
\end{definition}

Here, as in the rest of the paper, $\| \cdot \|_2$ is the $L^2$-norm $\|f\|_2 = \big(\int_{S^{d-1}} |f|^{2}\,d\mu\big)^{1/2}$. Of course, \Cref{def:L2-approx-designs} can be modified to account for \emph{weighted} approximate designs. For clarity, we'll stick with unweighted designs; small modifications of the proofs here imply the same results for weighted approximate designs.

If $\epsilon$ is small enough, then the lower bound in \Cref{thm:weighted-design-lower-bound} also holds for approximate designs.

\begin{proposition}\label{thm:lin-alg-approx-design-lower-bound}
    There is a constant $c_{d,t} > 0$ such that: If $\epsilon < c_{d,t}$, then any $\epsilon$-approximate spherical $2t$-design has at least $\dim(\mathcal P_t^\mu) = \binom{d+t -1}{d-1} + \binom{n+t-2}{d-1}$ points.
\end{proposition}
\begin{proof} 
    Because $\mathcal P_{2t}^\mu$ is a finite-dimensional vector space, all norms on $\mathcal P_{2t}^\mu$ are equivalent. So there is a constant $c_{d,t}$ such that $\|g\|_1 \geq c_{d,t}\|g\|_2$ for any $g \in \mathcal P_{2t}^\mu$. If $|X| < \dim(\mathcal P_{t}^\mu)$, then there is a nonzero polynomial $f$ of degree at most $t$ that vanishes on every point of $X$, in which case
    \[
        \left\lvert \frac{1}{|X|} \sum_{x \in X} f(x)^2 - \int\limits_{\mathclap{S^{d-1}}} f(x)^2\, dx\right\rvert
        = \lVert f^2\rVert_1
        \geq c_{t,d} \|f^2\|_2.
    \]
    Since $\epsilon < c_{d,t}$, condition \eqref{eq:L2-approx-design} fails for the polynomial $f^2$, so $X$ is not an $\epsilon$-approximate design.
\end{proof}

The remainder of this section is devoted to determining a better quantitative understanding of how the size of approximate designs depends on $\epsilon$ and $t$. To do so, we will formulate a linear programming bound, following Delsarte, Goethals, and Seidel's approach in \cite{delsarte-spherical-codes-designs}, with modifications to account for approximation. The linear programming bound appears as \Cref{thm:approx-design-LP}, and the quantitative bound on approximate designs is \Cref{thm:L2-approx-designs-lower-bound}.

Let $\hat{\mathcal P}_t^\mu$ denote the set of functions $f \in \mathcal P_t^\mu$ such that $\int_{S^{d-1}} f\,d\mu = 0$. A set $X$ is an $\epsilon$-approximate $t$-design if and only if 
\[
    \left\lvert \frac{1}{|X|} \sum_{x \in X} f(x)\right\rvert
        \leq \epsilon \lVert f\rVert_2
\]
for every $f \in \hat{\mathcal P}_t^\mu$. We now focus on this vector space.

In proving our linear programming result, we will make use of a special class of polynomials called the \emph{Gegenbauer polynomials}. For each $d$, the Gegenbauer polynomials $\{Q_k^d\}_{k\geq 0}$ are a sequence of orthogonal polynomials where $Q_k^d$ has degree $k$.\footnote{They are orthogonal with respect to the measure $(1-x^2)^{(d-3)/2}$ on the interval $[-1,1]$, though that specific fact won't be important to us.} A few relevant properties of the Gegenbauer polynomials are outlined here; further details are included in Section \ref{sec:gegenbauer-appendix}.

The evaluation map map $f \mapsto f(x)$ is a linear functional on the vector space $\hat{\mathcal P}_{t}^\mu$, so there is a polynomial $\ev_x \in \hat{\mathcal P}_{t}^\mu$ such that $\langle \ev_x, f\rangle = f(x)$, where the inner product is defined as $\langle f,g\rangle = \int_{S^{d-1}} fg\, d\mu$. As it turns out,
\begin{equation}\label{eq:Gegenbauer-angle}
    \langle \ev_x,\ev_y\rangle = \sum_{k=1}^t Q_k^d\!\big(\langle x,y\rangle).
\end{equation}
An important property of the Gegenbauer polynomials is that they are \emph{positive-definite kernels}, which in particular guarantees that for any finite point set $X \subset S^{d-1}$ and $k\geq 0$, we have
\[
    \sum_{x,y \in X} Q_k^d\!\big(\langle x,y\rangle\big) \geq 0.
\]

\begin{lemma}\label{thm:approx-design-condition}
    A set $X \subseteq S^{d-1}$ is an $\epsilon$-approximate spherical $t$-design if and only if
    \[
        \frac{1}{|X|^2} \sum_{x,y \in X} Q_{\leq t}^d\big(\langle x,y\rangle\big) \leq \epsilon^2,
    \]
    where $Q_{\leq t}^d = Q_1^d + Q_2^d + \cdots Q_t^d$.
\end{lemma}
\begin{proof}
    The point set $X$ is an approximate design if and only if 
    \[
        \left\lvert \Big\langle \frac{1}{|X|} \sum_{x\in X} \ev_x,\, f \Big\rangle \right\rvert 
        < \epsilon \lVert f\rVert_2
    \]
    for every $f \in \hat{\mathcal P}_{t}^\mu$. This is true if and only if
    \[
        \Big\rVert \frac{1}{|X|} \sum_{x\in X} \ev_x \Big\rVert_2 \leq \epsilon.
    \]
    Squaring both sides and applying \eqref{eq:Gegenbauer-angle} finishes the proof.
\end{proof}

\begin{lemma}\label{thm:approx-design-LP}
    Let $g = \sum_{k\geq 0} \alpha_k Q_k^d$ be a polynomial such that $g(s)\geq 0$ for $s \in [-1,1]$ and $\alpha_k \leq 0$ for $k > t$. Let $\alpha = \max_{1 \leq k \leq t} \alpha_k$. Any $\epsilon$-approximate spherical $t$-design has at least
    \[
        \frac{g(1)}{\alpha_0 + \epsilon^2 \alpha}
    \]
    points.
\end{lemma}
\begin{proof}
    As is typical in Delsarte-style linear programming bounds, we bound the sum $\frac{1}{|X|^2} \sum_{x,y \in X} g(\langle x,y\rangle)$ both above and below. For the upper bound, we have
    \begin{align*}
        \frac{1}{|X|^2} \sum_{x,y \in X} g\big(\langle x,y\rangle\big)
      &\leq
        \alpha_0 + \frac{1}{|X|^2} \sum_{x,y\in X} \sum_{k = 1}^t \alpha_k Q_k^d\big(\langle x,y\rangle\big)\\
      &\leq
        \alpha_0 + \frac{1}{|X|^2} \sum_{x,y \in X} \alpha \, Q_{\leq t}^d\big(\langle x,y\rangle\big)\\
      &\leq
        \alpha_0 + \epsilon^2 \alpha.
    \end{align*}
    The first inequality is because $Q_k^d$ is a positive-definite kernel, so $\sum_{x,y\in X} \alpha_k Q_k^d (\langle x,y\rangle) \leq 0$ for $k > t$. The second inequality is because $\sum_{k=0}^t (\alpha - \alpha_k) Q_k^d$ is a positive-definite kernel. And the last is just an application of \Cref{thm:approx-design-condition}.

    On the other hand, since $g(s) \geq 0$ for all $s \in [-1,1]$, we can obtain a lower bound by only counting the contributions from the terms where $x=y$:
    \[
        \frac{1}{|X|^2} \sum_{x,y \in X} g\big(\langle x,y\rangle\big)
        \geq \frac{1}{|X|} g(1).
    \]
    Combining the lower and upper bounds, we have
    \[
        |X| \geq \frac{g(1)}{\alpha_0 + \epsilon^2 \alpha}. \qedhere
    \]
\end{proof}

The original linear programming bound in \cite{delsarte-spherical-codes-designs} states that if $g$ satisfies the conditions in \Cref{thm:approx-design-LP}, then any spherical $t$-design has at least $g(1)/\alpha_0$ points. This means that as $\epsilon \to 0$, the linear programming bound for approximate designs approaches the linear programming bound for exact designs.

To get a numerical lower bound on $2t$-designs, we will choose the particular function $g = (Q_t^d)^2$. To carry out the calculations, we will employ a useful ``linearization formula'' for Gegenbauer polynomials.

Let $C_k^\lambda (x)$ denote the special function
\[
    C_k^\lambda (x) = \frac{(2\lambda)_k}{\big(\lambda + \frac{1}{2}\big)_k} P^{(\lambda - 1/2, \lambda - 1/2)}_k(x),
\]
where $P^{(\lambda - 1/2, \lambda - 1/2)}_k (x)$ is the Jacobi polynomial and $(\lambda)_k = \lambda (\lambda+1) \cdots (\lambda+k-1)$. Our polynomial $Q_k^d$ is equal to $C_k^{(d-2)/2}$ up to a constant depending on $k$ \cite{delsarte-spherical-codes-designs}:
\[
    Q_k^d
    = \frac{d+2k-2}{d-2}\ C_k^{(d-2)/2}.
\]

\begin{lemma}[Gegenbauer linearization {\cite[Theorem 6.8.2]{andrews1999special}}]
    Using the shorthand $(\lambda)_k = \lambda (\lambda+1) \cdots (\lambda+k-1)$, we have
    \[
        C_m^\lambda (x) C_n^\lambda (x)
        = \sum_{k=0}^{\min(m,n)} a_{m,n}(k) \, C^\lambda_{m+n-2k} (x),
    \]
    where
    \[
        a_{m,n}(k) = \frac{
            (m+n+\lambda - 2k)\, (\lambda)_k\, (\lambda)_{m-k}\, (\lambda)_{n-k} \, (2\lambda)_{m+n-k}
        }{
            (m+n+\lambda - k)\, k!\, (m-k)!\, (n-k)!\, (\lambda)_{m+n-k}\, (2\lambda)_{m+n-2k}
        }.
    \]
\end{lemma}

Using this lemma, we can determine the Gegenbauer expansion of $(Q_t^d)^2$:

\begin{corollary}\label{thm:Gegenbauer-square-coeficients}
    The Gegenbauer expansion of $(Q_t^d)^2 = \sum_{k=0}^{2t} a_t(k)\, Q_k^d$ has $a_t(k) = 0$ if $k$ is odd and $a_t(k) = \Theta_t(d^{t-k/2})$ if $k$ is even.
\end{corollary}

We can now prove a quantitative lower bound on the size of approximate designs. (As with exact designs, if $X$ is an $\epsilon$-approximate $2t$-design, then $X \sqcup (-X)$ is an $\epsilon$-approximate $(2t+1)$-design.)

\begin{theorem}\label{thm:L2-approx-designs-lower-bound}
    Any $\epsilon$-approximate spherical $2t$-design in $\R^d$ has at least
    \[
        c_t \frac{d^{2t}}{d^t + \epsilon^2 d^{t-1}}
    \]
    points.
\end{theorem}
\begin{proof}
    The function $g = (Q_t^d)^2$ satisfies the requirements of \Cref{thm:approx-design-LP}. By \Cref{thm:gegenbauer-value-at-1}, we have $g(1) = Q_t^d(1)^2 = \Theta_t(d^{2t})$, while \Cref{thm:Gegenbauer-square-coeficients} says that $\alpha_0 = \Theta_t(d^t)$ and $\alpha = \Theta_t(d^{t-1})$.
\end{proof}

As for the upper bound:

\begin{proposition}[Construction of approximate $L^2$-designs]
    There is an unweighted $\epsilon$-approximate spherical $2t$-design with at most $O_t(\epsilon^{-2} d^{2t})$ points.
\end{proposition}
\begin{proof}
    Let $f_1 \equiv 1, f_2,\dots, f_r$ be an orthonormal basis for the set of polynomial functions on $S^{d-1}$ of degree at most $2t$. We will use the probabilistic method to find a set $Y$ such that
    \[
        \left\lvert \frac{1}{|Y|} \sum_{y \in Y} f_i(y) - \int\limits_{\mathclap{S^{d-1}}} f_i\, d\mu\right\rvert^2
        \leq \epsilon^2
    \]
    for every $1 \leq i \leq r$. Since $\|f_i\|_2 = 1$, this means that $Y$ satisfies inequality \eqref{eq:L2-approx-design} for the functions $f_1,\dots,f_r$. We will then show that this implies that \eqref{eq:L2-approx-design} holds for all polynomials of degree at most $2t$.
    
    To begin, let $X = \{x_1,\dots,x_k\}$ be a set of $k$ points on $S^{d-1}$ selected uniformly and independently at random. We will calculate 
    \[
        \E_X \Big[ \Big( \frac{1}{k} \sum_{x \in X} f_i(x) - \Sint\! f_i\,d\mu \big)^2 \Big].
    \]
    If $i=1$, the expression inside the expectation is identically 0, so we will assume $i > 1$. The square expands as
    \begin{equation}\label{eq:expanded-L2-square}
        \frac{1}{k^2} \E_X \Big[ \Big(\sum_{x \in X} f_i(x)\Big)^2 \Big]
            - \frac{2}{k} \E_X \Big[ \Big(\sum_{x \in X} f_i(x)\Big) \Big(\Sint\! f_i\,d\mu\Big) \Big]
            + \Big(\Sint\! f_i\,d\mu\Big)^2.
    \end{equation}
    Because the $f_i$ are orthonormal and $f_1 \equiv 1$, we have $\int f_i^2\, d\mu = 1$, and $\int f_i\, d\mu = 0$ if $i > 1$. Using linearity of expectation, the first term of this expression is
    \begin{align*}
        \frac{1}{k^2} \E_X \Big[ \sum_{1\leq u,v\leq k} f_i(x_u) f_i(x_v) \Big]
        &= \frac{1}{k^2} \E_X \Big[\sum_{1 \leq u \leq k} f_i(x_u)^2 + \sum_{\substack{1 \leq u,v\leq k \\ u\neq v}} f_i(x_u)f_i(x_v) \Big]\\
        &= \frac{1}{k^2} \Big[\sum_{1 \leq u \leq k} \Big(\Sint f_i^2 \,d\mu\Big) + \sum_{\substack{1 \leq u,v\leq k \\ u\neq v}} \Big(\Sint f_i\,d\mu\Big)^2 \Big]\\
        &= \frac{1}{k^2} [k + 0].
    \end{align*}
    The second and third terms of \eqref{eq:expanded-L2-square} are each equal to 0.
    
    So, if we sum over all $i > 1$, we conclude that
    \[
        \E_X \Big[ \sum_{i=2}^r \Big( \frac{1}{k} \sum_{x \in X} f_i(x) - \Sint f_i\,d\mu \Big)^2 \Big]
        = \frac{r-1}{k}.
    \]
    Therefore, there is a set of $k$ points $Y \subset S^{d-1}$ such that
    \begin{equation}\label{eq:L2-approx-orthonormal}
        \sum_{i=2}^r \Big( \frac{1}{k} \sum_{y \in Y} f_i(x) \Big)^2
        \leq \frac{r-1}{k},
    \end{equation}
    since $\int f_i\,d\mu = 0$ for $i > 1$. That concludes the first part of the proof.
    
    Now take any function $g = \sum_{i=1}^r \alpha_i f_i$ with $\|g\|_2 = 1$ (in other words, $\sum_{i=1}^r \alpha_i^2 = 1$). Using the Cauchy--Schwarz inequality and \eqref{eq:L2-approx-orthonormal}, we have
    \[
        \Big\lvert \frac{1}{k} \sum_{y \in Y} g(y) - \Sint g\,d\mu\Big\rvert^2
        = \Big\lvert \sum_{i=2}^r \alpha_i \Big(\frac{1}{k} \sum_{y\in Y}  f_i(y)\Big) \Big\rvert^2
        \leq
        \bigg(\sum_{i=2}^r \alpha_i^2\bigg) \bigg( \sum_{i=2}^r \Big(\frac{1}{k} \sum_{y \in Y} f_i(y)\Big)^2 \bigg)
        \leq \frac{r-1}{k}.
        % = \frac{1}{k^2} \sum_{2 \leq i,j\leq r}
        % \alpha_i\alpha_j \Big(\sum_{y \in Y} f_i(y)\Big)\Big(\sum_{y \in Y} f_j(y)\Big)\\
        % &\leq \frac{1}{k^2} \sum_{2 \leq i,j\leq r}
        % |\alpha_i\alpha_j| \Big|\sum_{y \in Y} f_i(y)\Big|\, \Big|\sum_{y \in Y} f_j(y)\Big|\\
        % &\leq \frac{r-1}{k} \sum_{2\leq i,j\leq r} |\alpha_i \alpha_j|\\
        % &\leq \frac{r-1}{k} \Big(\sum_{1 \leq i \leq r}.
    \]
    If we choose $k = (r-1)/\epsilon^2$, we conclude that
    \[
        \Big\lvert \frac{1}{k} \sum_{y \in Y} g(y) - \Sint g\,d\mu\Big\rvert^2
        \leq \epsilon^2
    \]
    for every polynomial $g$ of degree at most $2t$ and $\|g\|_2 = 1$. In other words, $Y$ is an $\epsilon$-approximate $2t$-design. The vector space of polynomials of degree at most $2t$ has $r = \Theta_t(d^{2t})$ dimensions, so  $Y$ has $(r-1)/\epsilon^2 = \Theta_t (d^{2t}/\epsilon^{2})$ points.
\end{proof}

% For approximate designs, the difference between a multiset design and one with all distinct points is essentially semantics: Perturbing the points by an arbitrarily small amount results in a design with distinct points and an approximation ratio arbitrarily close to $\epsilon$.

This argument is not specialized to the sphere at all: The same argument, nearly word-for-word, may be used to construct an approximate design for any set of functions over any measure space.

\subsection{Approximate designs via tensors}\label{sec:tensor-approximate-designs}

The defining condition \eqref{eq:weighted-design-condition} of designs can be phrased in terms of tensor products of vectors, and this alternative perspective will provide a different definition of approximate designs. Given a vector $x \in \R^d$, the entries of $x^{\otimes t}$ correspond to evaluations of monomials: $(x^{\otimes t})_{\alpha_1,\alpha_2,\dots,\alpha_t} = \prod_{i = 1}^t x_{\alpha_i}$. A weighted set $(X,w)$ is therefore a $t$-design if and only if
\[
    \E_{x \sim w} [x^{\otimes k}] = \E_{v \sim \mu} [v^{\otimes k}]
\]
for every integer $0\leq k \leq t$. (The constant monomial guarantees that $\sum_{x \in X} w(x) = 1$, so $w$ is indeed a probability measure.) Since $x \in S^{d-1}$, we have $x^\alpha = x^\alpha (x_1^2 + \cdots + x_d^2)$. Therefore the condition $\E_{x \sim w} [x^{\otimes k}] = \E_{v \sim \mu} [v^{\otimes k}]$ implies the condition $\E_{x \sim w} [x^{\otimes (k-2)}] = \E_{v \sim \mu} [v^{\otimes (k-2)}]$. In other words,

\begin{proposition}
    A weighted set $(X,w)$ is a spherical $t$-design if and only if $\E_{x \sim w} [x^{\otimes k}] = \E_{v \sim \mu} [v^{\otimes k}]$ for $k\in \{t-1,t\}$.
\end{proposition}

If $\E_{x\sim w} [x^{\otimes 2t}] = \E_{v \sim \mu} [v^{\otimes 2t}]$, the set $X\cup (-X)$ with the weight function $\frac{1}{2}\big( w(x) + w(-x)\big)$ is a $2t$-design. Thus, if we are willing to double the size of the design, we can ignore the $k=2t-1$ condition, which leads us to a different definition of an approximate design:

\begin{definition}
    An \emph{$\epsilon$-approximate spherical tensor $2t$-design} is a set $X$ of points with a probability measure $w\colon X \to \R_{>0}$ such that
    \[
        \Big\lVert \E_{x \sim w} [x^{\tensor 2t}] - \E_{v \sim \mu} [v^{\tensor 2t}] \Big\rVert_2
        \leq \epsilon.
    \]
\end{definition}

This definition parallels definitions of \emph{approximate unitary designs}, which are defined similarly and have been intensively studied by quantum computer scientists \cite{quant-des1,quant-des2,quant-des3,quant-des4}.

We now prove \Cref{thm:approx-tensor-des} in two parts, the lower and upper bounds.

\begin{proposition}
    Any $\epsilon$-approximate spherical tensor $2t$-design has at least $\epsilon^{-2} - o_{d\to\infty}(1)$ points.
\end{proposition}
\begin{proof}
    For any weighted set $(X,w)$, we have
    \begin{equation}\label{eq:tensor-approx-lower-bound}
        \Big\lVert \E_{x \sim w} [x^{\tensor 2t}] - \E_{v \sim \mu} [v^{\tensor 2t}] \Big\rVert_2^2
        = \E_{x\sim w,y \sim w} \langle x^{\otimes 2t}, y^{\otimes 2t}\rangle
        - 2 \E_{x \sim w, v \sim \mu} \langle x^{\otimes 2t}, v^{\otimes 2t}\rangle
        + \E_{u\sim \mu,v \sim \mu}\langle u^{\otimes 2t}, v^{\otimes 2t}\rangle.
    \end{equation}
    To prove the result, we will lower bound the first term using the contributions where $x=y$ and show that the other two terms are negligible.
    
    For any $u \in S^{d-1}$, we have (according to \Cref{thm:monomial-spherical-moment}):
    \[
        \E_{v \sim \mu} \langle u^{\otimes 2t}, v^{\otimes 2t}\rangle
        = \int x_1^{2t}\, dx 
        = \frac{(2t-1)!!}{d(d+2)\cdots(d+2t-2)}
        = \Theta_t(d^{-t}).
    \]
    So the last two terms in \eqref{eq:tensor-approx-lower-bound} are of size $O_t(d^{-t})$. Because
    \[
    \langle x^{\otimes 2t}, y^{\otimes 2t}\rangle
    = \langle x,y\rangle^{2t}
    \geq 0,
    \]
    we can lower bound $\E_{x,y \sim w} \langle x^{\otimes 2t}, y^{\otimes 2t}\rangle$ by taking only the terms where $x=y$ and applying Cauchy-Schwarz:
    \[
        \E_{x,y \sim w} \langle x^{\otimes 2t}, y^{\otimes 2t}\rangle
        \geq \sum_{x \in X} w(x)^2
        \geq \frac{1}{|X|}.
    \]
    Putting this all together, we get
    \[
        \Big\lVert \E_{x \sim w} [x^{\tensor 2t}] - \E_{v \sim \mu} [v^{\tensor 2t}] \Big\rVert^2
        \geq
        \frac{1}{|X|} - \Theta_t(d^{-t}).
    \]
    Since $X$ is an approximate design, we conclude that $|X|^{-1} \leq \epsilon^2 + \Theta_t(d^{-t})$; therefore $|X| \geq \epsilon^{-2} - o(1)$.
\end{proof}

Surprisingly, this lower bound is asymptotically tight!

\begin{proposition}[Construction of approximate tensor designs]\label{thm:approx-tensor-design-construction}
    There is an unweighted $\epsilon$-approximate spherical tensor $2t$-design with at most $\epsilon^{-2}$ points.
\end{proposition}
\begin{proof}
    We will use the probabilistic method. For each $\beta \in \{1,2,\dots,d\}^{2t}$, let $f_{\beta}(x) = \prod_{j=1}^{2t} x_{\beta_j}$, so that $x^{\otimes 2t} = \big( f_\beta(x)\big)_{\beta \in [d]^{2t}}$. A set $X$ is an $\epsilon$-approximate spherical tensor $2t$-design if and only if
    \begin{equation}\label{eq:expanded-tensor-def}
        \Big\lVert \frac{1}{k} \sum_{x \in X} x^{\otimes 2t} - \Sint x^{\otimes 2t}\Big\rVert_2^2
        = \sum_{\beta \in [d]^{2t}} \Big[ \Big(\frac{1}{k} \sum_{x \in X} f_{\beta}(x) - \int_{S^{d-1}} f_\beta\, d\mu \Big)^2\Big]
        \leq \epsilon^2.
    \end{equation}
    Let $X = \{x_1,\dots,x_k\}$ be a set of $k$ uniform, independent points on $S^{d-1}$. We will show that the expected value of the left hand side for a set of $\epsilon^{-2}$ random points on $S^{d-1}$ is at most $\epsilon^2$.

    Let $X = \{x_1,\dots,x_k\}$ be a set of independent, uniformly distributed random points on $S^{d-1}$. We have:
    \begin{align}
        \E_X \Big[ \sum_{\beta \in [d]^{2t}} \Big(\frac{1}{k} \sum_{x \in X} f_{\beta}(x) &- \Sint  f_\beta\, d\mu \Big)^2\Big]\notag\\
        &= \E_X \bigg[
            \sum_{\beta \in [d]^{2t}} \Big[ \frac{1}{k^2} \Big(\sum_{x \in X} f_{\beta}(x)\Big)^2
            - \frac{2}{k}\Big(\sum_{x \in X} f_{\beta}(x)\Big) \Big(\Sint  f_\beta\, d\mu \Big) 
            + \Big(\Sint  f_\beta\, d\mu \Big)^2\Big]
            \bigg]\notag\\
        &= \sum_{\beta \in [d]^{2t}}\bigg(
            \frac{1}{k^2} \E_X  \Big[\sum_{x \in X} f_{\beta}(x)\Big]^2
            - \Big(\Sint  f_\beta\, d\mu \Big)^2\bigg)\label{eq:expand-approximate-tensor-expectation}
    \end{align}
    The remaining expectation simplifies as
    \begin{align*}
        \frac{1}{k^2} \E_X  \Big[\sum_{x \in X} f_{\beta}(x)\Big]^2
        &= \frac{1}{k^2} \E_X  \Big[\sum_{1\leq u\leq k} f_{\beta}(x_u)^2 + \sum_{\substack{1 \leq u,v\leq k \\ u\neq v}} f_\beta(x_u)f_\beta(x_v)\Big]\\
        &= \frac{1}{k^2} \Big[k \Sint f_{\beta}^2\, d\mu + k(k-1) \Big(\Sint f_\beta\, d\mu\Big)^2\Big]\\
        &= \frac{1}{k} \Sint f_{\beta}^2\, d\mu + \frac{k-1}{k} \Big(\Sint f_\beta\, d\mu\Big)^2.
    \end{align*}
    Substituting this into \eqref{eq:expand-approximate-tensor-expectation}, we get
    \begin{align*}
        \E_X \Big[ \sum_{\beta \in [d]^{2t}} \Big(\frac{1}{k} \sum_{x \in X} f_{\beta}(x) - \Sint  f_\beta\, d\mu \Big)^2\Big]
        &=  \sum_{\beta \in [d]^{2t}} \bigg(
            \frac{1}{k} \Sint f_{\beta}^2\, d\mu 
            - \frac{1}{k} \Big(\Sint  f_\beta\, d\mu \Big)^2\bigg)\\
        &\leq \frac{1}{k} \sum_{\beta \in [d]^{2t}}
            \ \Sint f_{\beta}^2\, d\mu.
    \end{align*}

    The final step is to notice that $\sum_{\beta \in [d]^{2t}} f_\beta^2 = (x_1^2 + x_2^2 + \dots + x_d^2)^{2t}$. But $x_1^2 + x_2^2 + \cdots + x_d^2 = 1$ for every $x \in S^{d-1}$, so
    \begin{align*}
        \E_X \Big[ \sum_{\beta \in [d]^{2t}} \Big(\frac{1}{k} \sum_{x \in X} f_{\beta}(x) - \Sint  f_\beta\, d\mu \Big)^2\Big]
        \leq \frac{1}{k} \sum_{\beta \in [d]^{2t}}
        \ \Sint f_{\beta}^2\, d\mu
        = \frac{1}{k} \Sint (x_1^2 + x_2^2 + \cdots + x_d^2)^{2t}\, d\mu
        = \frac{1}{k}.
    \end{align*}

    Since the average is at most $1/k$, there is a set of $k$ points whose average is at most $1/k$. By taking $k=\epsilon^{-2}$, this set of $\epsilon^{-2}$ points forms an $\epsilon$-approximate tensor $2t$-design.
\end{proof}
% \begin{proof}
%     Let $z = \int_{S^{d-1}} v^{\otimes 2t}\, d\mu$. Since $\| v^{\otimes 2t}\| = 1$ for every $v \in S^{d-1}$, the diameter of $X = \{v^{\otimes 2t} : v \in S^{d-1}\}$ is at most 2. By \Cref{thm:no-dim-caratheodory}, there is a multiset $Y \subset S^{d-1}$ of $k$ points such that $\lVert z - \operatorname{cent}(Y)\rVert \leq 2/\sqrt{2k}$. Taking $k = 2\epsilon^{-2}$, we find that
%     \[
%         \Big\| \frac{1}{|Y|} \sum_{y \in Y} \, y^{\otimes 2t} - \int\limits_{\mathclap{S^{d-1}}} v^{\otimes 2t}\,d\mu\ \Big\|
%         \leq \epsilon,
%     \]
%     so $Y$ is an $\epsilon$-approximate tensor $2t$-design.
% \end{proof}

As a final note of caution, an $\epsilon$-approximate tensor $2t$-design is not necessarily an $\epsilon$-approximate tensor $(2t-2)$-design. To construct a set that is simultaneously a 2-, 4-, $\dots$, $2t$-design, we can imitate the previous argument but sum over all $\beta \in [d]^{2s}$ for $1 \leq s \leq t$. The resulting design has at most $t \epsilon^{-2}$ points.

\section{Open questions}\label{sec:open}

There are many remaining questions for fixed-strength spherical designs, the most prominent of which is determining the size of the smallest spherical designs. One might guess that the linear programming lower bound of \Cref{thm:delsarte-lower-bound} is tight:

\begin{question}\label{q:small-sph-designs}
    Is there a weighted spherical $2t$-design in $\R^d$ with $O_t(d^t)$ points?
\end{question}

There are several suggestive, though circumstantial, reasons to believe the answer is ``yes''. First, there are signed $2t$-designs with $O_t(d^t)$ points for every strength, which beats the degrees-of-freedom heuristic and indicates the same may be true for weighted or even unweighted designs. Moreover, these signed designs simultaneously average \emph{all} monomials with an odd degree in any variable, of any degree, which suggests that only monomials that have even degree in every variable significantly impact the size of the design. If that's true, then the expected size of a $2t$-design would in fact be $O_t(d^t)$. And, of course, the answer is ``yes'' for $t=2$ and $t=4$.

One method to improve the upper bound on designs it to construct a smaller $t$-wise independent set. As mentioned in the introduction, there are several well-known constructions. One of the most common constructions comes from Reed--Solomon codes and yields a $t$-wise independent subset of $\{1,2,\dots,q-1\}^q$ of size $O_{t,q}(d^t)$ (see, for example, \cite[Section 5.5]{orthogonal-arrays-book}). Alon, Babai, and Itai produced a $(2r+1)$-wise independent set in $\{1,2\}^d$ with $O_t(d^{r})$ points \cite{k-wise-independence}. (See Section 15.2 of \cite{alon-spencer} for an exposition that doesn't require knowledge of BCH codes.) Extending their proof from $\F_2$ to $\F_q$ produces a $(qr+1)$-wise independent set in $\{1,2,\dots,q\}^d$ with $d^{(q-1)r}$ points, which significantly improves on the Reed--Solomon construction when $t > q$.

However, in the proof of \Cref{thm:unweighted-gaussian-design}, $q$ is the size of an unweighted Gaussian $t$-design for $\R^1$. One can check via computer, using the Gerard--Newton formulas, that there is no such design with $t$ points for small $t \geq 4$ (my program checked $4 \leq t \leq 500$), and this presumably holds for all $t$. Since $t < q$, the Alon--Babai--Itai construction also produces a set with $O_t(d^t)$ points; so \Cref{thm:t-wise-ind} is more effective for this application.

One approach to constructing even smaller $t$-wise independent sets is to find a larger set of $t$-wise linearly independent vectors in $\F_q^r$. \Cref{thm:t-wise-linearly-indep-set} finds a set of size $\frac{1}{8q}(q^r)^{1/(t-1)}$. If a set of $c_{q,t}\, (q^r)^{\beta(t)}$ vectors in $\F_q^r$ (for fixed $q$ and large enough $r$) were found, then substituting that result for \Cref{thm:t-wise-linearly-indep-set} in the proof of \Cref{thm:unweighted-gaussian-design} would immediately produce a weighted spherical $2t$-design with $O_t(d^{1/\beta(2t)})$ points in $\R^d$.

\begin{question}
    What is the size of the largest subset of $\F_q^r$ that does not contain $t+1$ linearly dependent vectors?
\end{question}

Conversely, an upper bound on the size of $t$-wise linearly independent sets limits the potential success of this approach:

\begin{proposition}\label{thm:linearly-indep-vectors-upper-bound}
    If $S \subseteq \F_q^r$ does not contain a set of $t$ linearly dependent vectors, then $|S| \leq C_t\, (q^r)^{2/t}$.
\end{proposition}
\begin{proof}
    Since $S$ has no nontrivial linear dependence of size $t$, each of the vectors $v_1 + \cdots + v_{t/2}$ with $v_i \in S$ must be distinct. There are $\binom{|S|}{t/2}$ such vectors, so $q^r \geq \binom{|S|}{t/2} \geq c_t |S|^{t/2}$. 
\end{proof}

\Cref{thm:linearly-indep-vectors-upper-bound} implies that this approach cannot produce a spherical $2t$-design with fewer than $\Theta_t(d^{t})$ points---which, of course, we already knew. However, a better upper bound in \Cref{thm:linearly-indep-vectors-upper-bound} would show that this approach cannot affirmatively answer \Cref{q:small-sph-designs}.

% A similar direction is to improve the bounds on the size of approximate designs. While we have nearly tight bounds on number of points in smallest approximate tensor designs, the bounds on approximate $L^2$-designs are not as tight.

% \begin{problem}
%     Improve the upper or lower bound for  approximate $L^2$-designs.
% \end{problem}

% The upper bounds in \Cref{sec:approximate-designs} produce weighted approximate designs, so a quantitative upper bound for optimal \emph{unweighted} approximate designs also remains open.

All the upper bound proofs in this paper assert the existence of a design but don't produce a specific set, and previous constructions \cite{3-designs,3-designs-abelian,5-designs} are for $t \leq 5$. It would be nice to find more families of explicit constructions:

\begin{problem}
    Provide an explicit construction of spherical $t$-designs with few points for $t \geq 6$.
\end{problem}

The original lower bound on the size of designs, in \cite{delsarte-spherical-codes-designs}, relied on the linear programming method, as does our proof of the related result for approximate designs (\Cref{thm:L2-approx-designs-lower-bound}). \Cref{thm:weighted-design-lower-bound} provides a simple linear-algebraic proof of the lower bound for exact designs.

\begin{question}
    Is there a purely linear-algebraic proof of a lower bound for approximate designs that is comparable to \Cref{thm:L2-approx-designs-lower-bound}?
\end{question}

While we determined the asymptotic size of tensor-approximate spherical designs, the upper and lower bounds for $L^2$-approximate designs are further apart.

\begin{problem}
    Improve the upper or lower bound for  approximate $L^2$-designs.
\end{problem}

Finally, in \Cref{sec:gaussian-designs}, we used a spherical $t$-design to produce a Gaussian $t$-design, and vice versa. But this technique only works for weighted designs. If unweighted Gaussian designs can be transferred to unweighted spherical designs and vice versa, we could project \emph{unweighted} spherical designs to lower dimensions, just as \Cref{thm:spherical-design-projection} allows us to project weighted spherical designs.

\begin{question}
    Is there a constant $c_t$ such that the existence of an unweighted spherical $t$-design with $N$ points implies the existence of an unweighted Gaussian $t$-design with at most $c_t N$ points? What about converting Gaussian designs to spherical ones?
\end{question}

\vspace{1.4\baselineskip}
\noindent
\begin{minipage}{\textwidth}
    \noindent\hspace*{\fill}
        {\large\scshape acknowledgments}
    \hspace*{\fill}\par
    \vspace{0.75\baselineskip}
    
    \noindent
    \hspace*{\fill}
    \begin{minipage}{0.95\textwidth}
        I thank Henry Cohn and Yufei Zhao especially for their many insightful and delightful conversations; Noga Alon for a discussion on $t$-wise independent sets; Ayodeji Lindblad for several conversations on various parts of this paper; Lisa Sauermann for the proof of \Cref{thm:Sauermann}; Xinyu Tan for helpful pointers on the unitary design literature; Hung-Hsun Hans Yu for communicating a construction of spherical 4-designs. I also thank the anonymous reviewer, whose suggestions greatly improved the exposition throughout the paper and improved the upper bounds for approximate designs. This work was partially supported by a National Science Foundation Graduate Research Fellowship under Grant No. 2141064.
    \end{minipage}
    \hspace*{\fill}
\end{minipage}
\vspace{0.6\baselineskip}

\let\OLDthebibliography\thebibliography
\renewcommand\thebibliography[1]{
  \OLDthebibliography{#1}
  \setlength{\parskip}{0pt}
  \setlength{\itemsep}{2pt plus 0.3ex}
}

\addcontentsline{toc}{section}{References}

{
    \small
    \setstretch{1}
    \bibliographystyle{amsplain-nodash}
    \bibliography{bibliography}
}

\appendix
\section{Appendix}

\subsection{Spherical moments of monomials}

For every odd integer $k$, define $k!! = k(k-2)\cdots 3 \cdot 1$.

\begin{proposition}\label{thm:monomial-spherical-moment}
    If $k_1,\dots,k_d$ are even nonnegative integers, then
    \[
        \int\limits_{\mathclap{S^{d-1}}} x_1^{k_1}\cdots x_d^{k_d}\, d\mu
        = \frac{\prod_{i=1}^d (k_i-1)!!}{d(d+2)\cdots(d+k-2)}.
    \]
    If any of $k_1,\dots,k_d$ is odd, then $\int_{S^{d-1}} x_1^{k_1}\cdots x_d^{k_d}\,d\mu = 0$.
\end{proposition}
\begin{proof}
    If $k_i$ is odd, then the symmetry $x_i \mapsto -x_i$ (reflection over a coordinate hyperplane) shows that
    \[
        \int\limits_{\mathclap{S^{d-1}}} x_1^{k_1}\cdots x_d^{k_d}\,d\mu
        = -\int\limits_{\mathclap{S^{d-1}}} x_1^{k_1}\cdots x_d^{k_d}\,d\mu,
    \]
    so the integral vanishes.

    For the rest of the proof, we assume that $k_1,\dots,k_d$ are all even. Let $\sigma_d$ be the surface area of $S^{d-1}$ with respect to the Lebesgue measure. To evaluate the spherical moment of a monomial, we integrate it against a Gaussian as in \eqref{eq:spherical-to-gaussian-integral}, which then splits into the product of several single-variable integrals:
    \begin{equation}\label{eq:spherical-gaussian-integral}
        \sigma_d\int\limits_{\mathclap{S^{d-1}}} x_1^{k_1}\cdots x_d^{k_d}\, d\mu\ \int_0^\infty r^{k_1 + \cdots + k_d} e^{-\pi r^2}r^{d-1}\, dr
        \ =\ \int\limits_{\mathclap{\R^d}} x_1^{k_1} \cdots x_d^{k_d} e^{-\pi|x|^2}\, d\rho
        \ =\ \prod_{i=1}^d \int_{-\infty}^\infty x^{k_i} e^{-\pi x^2}\, dx.
    \end{equation}
    To integrate the Gaussians, set $y = \pi x^2$; then $x\, dx = \frac{1}{2\pi}\, dy$ and
    \[
        \int_0^\infty x^k e^{-\pi x^2}\, dx
        = \frac{1}{2\pi \cdot \pi^{(k-1)/2}} \int_0^\infty y^{(k-1)/2} e^{-y}\, dy
        = \frac{1}{2\pi^{(k+1)/2}}\, \Gamma\Big(\frac{k+1}{2}\Big).
    \]
    If $k$ is even, then we have
    \begin{equation}\label{eq:gaussian-monomial-integral}
        \int_{-\infty}^\infty x^k e^{-\pi x^2}\, dx
        = 2\int_0^\infty x^k e^{-\pi x^2}\, dx
        = \frac{1}{\pi^{(k+1)/2}}\, \Gamma\Big(\frac{k+1}{2}\Big).
    \end{equation}
    
    Combining \eqref{eq:spherical-gaussian-integral} and \eqref{eq:gaussian-monomial-integral}, and setting $k := \sum_{i=1}^d k_i$, we have
    \[
        \sigma_d \int\limits_{\mathclap{S^{d-1}}} x_1^{k_1}\cdots x_d^{k_d}\, d\mu
        = \frac{ \pi^{-(k+d)/2}\, \prod_{i=1}^d \Gamma\Big(\frac{k_i+1}{2}\Big)}{(2 \pi^{(k + d)/2})^{-1}\ \Gamma\Big(\frac{k+d}{2}\Big)}
        = \frac{2\prod_{i=1}^d \Gamma\Big(\frac{k_i+1}{2}\Big)}{\Gamma\Big(\frac{k+d}{2}\Big)}.
    \]
    
    Taking $k_1 = \cdots = k_d = 0$, we find that
    \[
        \sigma_d = \frac{2\, \Gamma\Big(\frac{1}{2}\Big)^d}{\Gamma\Big(\frac{d}{2}\Big)},
    \]
    so
    \[
        \int\limits_{\mathclap{S^{d-1}}} x_1^{k_1}\cdots x_d^{k_d}\, d\mu
        = \frac{\Gamma\Big(\frac{d}{2}\Big) \prod_{i=1}^d \Gamma\Big(\frac{k_i+1}{2}\Big)}{\Gamma\Big(\frac{1}{2}\Big)^d \Gamma\Big(\frac{k+d}{2}\Big)}.
    \]
    Using the fact that $\Gamma(x+1) = x\,\Gamma(x)$, the previous equation simplifies to
    \[
        \int\limits_{\mathclap{S^{d-1}}} x_1^{k_1}\cdots x_d^{k_d}\, d\mu
        = \frac{\prod_{i=1}^d (k_i-1)!!}{d(d+2)\cdots(d+k-2)}.\qedhere
    \]
\end{proof}

\subsection{Gegenbauer polynomials}\label{sec:gegenbauer-appendix}

The orthogonal group $O(d)$ acts on the vector space $\mathcal P_t^\mu$ via its typical action on the sphere: $(U\cdot f)(x) = f(U^{-1}x)$. With this action, $\mathcal P_t^\mu$ is an $O(d)$-representation, and it has the irreducible decomposition
\[
    \mathcal P_t^\mu = \bigoplus_{k=0}^t \mathcal W_k,
\]
where $\mathcal W_k$ is the vector space of harmonic polynomials that are homogeneous of degree $k$ restricted to the sphere. (A polynomial $f$ is \emph{harmonic} if $\Delta f \equiv 0$, where $\Delta = (\frac{\partial^2}{\partial x_1^2} + \cdots + \frac{\partial^2}{\partial x_d^2})$.)

Let $\mathcal Q_k$ denote the space of all polynomials of degree at most $k$. Since $\Delta\colon \mathcal Q_k \to \mathcal Q_{k-2}$ and $\mathcal W_k = \ker \Delta$, we see that $\dim(\mathcal W_k) \geq \dim(\mathcal Q_k) - \dim(\mathcal Q_{k-2})$. In fact, equality holds, so
\[
    \dim(\mathcal W_k)
    = \binom{d+k-1}{d-1} - \binom{d+k-3}{d-1}.
\]
Summing over $k$, we find that
\[
    \dim(\mathcal P_t^\mu) = \binom{d+t-1}{d-1} + \binom{d+t-2}{d-1}.
\]
For a full proof of these assertions, see Section 3.3 of Henry Cohn's notes \cite{henry-fourier-notes}.

Gegenbauer polynomials arise from these irreducible representations. For each $x \in S^{d-1}$, the map $f\mapsto f(x)$ is a linear functional on $\mathcal W_k$, so there is a polynomial $\ev_{k,x} \in \mathcal W_k$ such that $f(x) = \langle f, \ev_{k,x}\rangle$ for every $f \in \mathcal W_k$. (The polynomial $\ev_x$ in \Cref{sec:L2-approximate-designs} is $\ev_x = \sum_{k=1}^t \ev_{k,x}$.) Since
\[
    \ev_{k,x}(y) = \langle \ev_{k,x},\ev_{k,y}\rangle
    = \ev_{k,y}(x),
\]
the evaluation polynomials are symmetric in $x$ and $y$.

As it turns out, the inner product $\langle \ev_{k,x}, \ev_{k,y}\rangle$ is invariant under the action of the orthogonal group on $x$ and $y$. For any $U \in O(d)$ and $f \in \mathcal W_k$,
\[
    \langle f, \ev_{k, Ux}\rangle 
    = f(Ux)
    = (U^{-1}\cdot f) (x)
    = \langle f, U \cdot \ev_{k,x}\rangle,
\]
since $U^{-1} = U^\top$. As equality holds for every $f \in \mathcal W_k$, we conclude that $\ev_{k,Ux} = U\cdot \ev_{k,x}$. Thus
\[
    \langle \ev_{k,x},\ev_{k,y}\rangle
    = \langle U\cdot \ev_{k,x}, U \cdot \ev_{k,y}\rangle
    = \langle \ev_{k,Ux}, \ev_{k,Uy}\rangle.
\]

As a result, the value of $\langle \ev_{k,x},\ev_{k,y}\rangle$ is determined entirely by the inner product of $x$ and $y$:

\begin{definition}
    The \emph{Gegenbauer polynomial} $Q_k^d$ is defined by
    \[
        Q_k^d(\langle x,y\rangle) = \langle \ev_{k,x},\ev_{k,y}\rangle.
    \]
\end{definition}

Alternatively, the Gegenbauer polynomials may be defined inductively (as in \cite{delsarte-spherical-codes-designs}), but that approach doesn't provide any geometric intuition.

\begin{proposition}\label{thm:gegenbauer-value-at-1}
    $Q_k^d(1) = \dim(\mathcal W_k^d) = \binom{d+k-1}{d-1} - \binom{d+k-3}{d-1}$.
\end{proposition}
\begin{proof}
    The linear transformation $\ev_{k,x}\ev_{k,x}^{\top}\colon f\mapsto \langle \ev_x, f\rangle\, \ev_x$ has trace
    \[
        \Tr(\ev_x\ev_x^{\top})
        = \Tr(\ev_x^\top \ev_x)
        = \langle \ev_x, \ev_x\rangle
        = Q_k^d(1).
    \]
    The linear transformation
    \[
        E := \int\limits_{\mathclap{S^{d-1}}} \ev_{k,x}\ev_{k,x}^\top\ d\mu(x)
    \]
    thus also has trace $Q_k^d(1)$. We claim that $E$ is in fact the identity operator on $\mathcal W_k$. Given any polynomial $f$, we have
    \[
        Ef = \int\limits_{\mathclap{S^{d-1}}} \ev_{k,x}\ev_{k,x}^{\top} \, f\ d\mu(x)
        = \int\limits_{\mathclap{S^{d-1}}} f(x)\, \ev_{k,x} \ d\mu(x).
    \]
    Therefore
    \[
        \big(Ef\big) (y) 
        = \int\limits_{\mathclap{S^{d-1}}} f(x)\, \ev_{k,x}(y) \ d\mu(x)
        = \int\limits_{\mathclap{S^{d-1}}} f(x)\, \ev_{k,y}(x) \ d\mu(x)
        = f(y),
    \]
    and $Ef = f$; so
    \[
        Q_k^d(1)
        = \Tr(E)
        = \dim(\mathcal W_k)
        = \binom{d+k-1}{d-1} - \binom{d+k-3}{d-1}
        .\qedhere
    \]
\end{proof}

\vspace{2\baselineskip}
\noindent
{\small \textsc{Travis Dillon}}\\
{\small \textsc{Department of Mathematics, Massachusetts Institute of Technology, Cambridge, MA, USA}}\\
\textit{email:} \texttt{dillont@mit.edu}

\end{document}